\documentclass[11pt]{amsart}

\usepackage{latexsym, enumitem}
\usepackage{rotating}
\usepackage{amsfonts}
\usepackage{amssymb}
\usepackage{amsmath, amsthm}
\usepackage{graphics, color}
\definecolor{darkblue}{rgb}{0,0,0.4}
\usepackage[colorlinks=true, citecolor=darkblue, filecolor=darkblue,
linkcolor=darkblue, urlcolor=darkblue]{hyperref}
\input xy
\xyoption{all}
\DeclareMathOperator*{\colim}{colim}
\DeclareMathOperator*{\tele}{telescope}

\DeclareMathOperator*{\cc}{\Box}

\newtheorem{theorem}{Theorem}[section] % 1st argument is your name for it
\newtheorem{lemma}[theorem]{Lemma}     % 2nd argument is what is printed
\newtheorem{corollary}[theorem]{Corollary}
\newtheorem{proposition}[theorem]{Proposition}

\theoremstyle{definition}
\newtheorem{definition}[theorem]{Definition}
\newtheorem{remark}[theorem]{Remark}
\newtheorem{question}[theorem]{Question}

\newtheorem{examples}[theorem]{Examples}

\newcommand{\Top}{{\bf CGTop}}

\newcommand{\lra}[1]{\langle #1 \rangle}

\newcommand{\srou}[2]{\underset{#2}{\overset{#1}{\maps}}}
\newcommand{\srlou}[2]{\underset{#2}{\overset{#1}{\longleftarrow}}}
\newcommand{\br}[1]{\left[ #1 \right]}

\renewcommand{\br}[1]{\left[ #1 \right]}
\newcommand{\bt}{\bullet}

\newcommand{\rh}{\widetilde{\Pi}} 
\newcommand{\ul}{\underline}

\newcommand{\bbC}{\mathbb{C}}

\newcommand{\N}{\mathbb{N}}
\newcommand{\bbN}{\mathbb{N}}

\newcommand{\bbR}{\mathbb{R}}

\newcommand{\bbZ}{\mathbb{Z}}

\newcommand{\bbQ}{\mathbb{Q}}
\newcommand{\ignore}[1]{}

\newcommand{\mC}{\mathcal{C}}
\newcommand{\qcd}{\bbQ cd}

\newcommand{\bS}{\mathbf{S}}

\newcommand{\leqs}{\leqslant}
\newcommand{\geqs}{\geqslant}
\newcommand{\heq}{\simeq}

\newcommand{\maps}{\longrightarrow}

\newcommand{\injects}{\hookrightarrow}
\newcommand{\homeo}{\cong}

\newcommand{\isom}{\cong}
\newcommand{\cross}{\times}

\newcommand{\wt}[1]{\widetilde{#1}} %wide tilde for M's
\newcommand{\Rdef}{R^{\mathrm{def}}}

\newcommand{\Rep}{\mathrm{Rep}}

\newcommand{\Hom}{\mathrm{Hom}}

\newcommand{\GL}{\mathrm{GL}}
 
\newcommand{\U}{\mathrm{U}}
\newcommand{\hU}{h\mathrm{U}}
\newcommand{\hGL}{h\mathrm{GL}}
\newcommand{\K}{K^{\mathrm{def}}}

\newcommand{\hofib}{\mathrm{hofib}}
\newcommand{\Stab}{\mathrm{Stab}}

\newcommand{\Map}{\mathrm{Map}}

\newcommand{\Gr}{\mathrm{Gr}}

\newcommand{\ku}{\mathbf{ku}}
\newcommand{\Susp}{\Sigma}
\newcommand{\Id}{\mathrm{Id}}

\newcommand{\xmaps}{\xrightarrow}
\newcommand{\srm}[1]{\stackrel{#1}{\maps}}
\newcommand{\srt}[1]{\stackrel{#1}{\to}}
\newcommand{\sm}{\wedge}

\newcommand{\goesto}{\mapsto}
\newcommand{\nd}{\noindent}

\newcommand{\ol}[1]{\overline{#1}}
\def\co{\colon\thinspace}

\newcommand{\Img}{\mathrm{Im}}
 \newcommand{\s}[1]{\vspace{#1 in}}
\newcommand{\e}{\emph}

 \newcommand{\ra}{\rangle}
\newcommand{\la}{\langle}
 
\newcommand{\defn}{\mathrel{\mathop :}=}

\title[The homotopy groups of a homotopy group completion]% end with percent
 {The homotopy groups of a homotopy group completion}

\author{Daniel A. Ramras}


\thanks{The author was partially supported by the Simons Foundation (\#279007)}

\begin{document}

\begin{abstract} Let $M$ be a topological monoid with homotopy group completion $\Omega BM$.
Under a strong homotopy commutativity hypothesis on $M$, we show that $\pi_k (\Omega BM)$ 
is the quotient of the monoid of free homotopy classes $[S^k, M]$ by its submonoid of nullhomotopic maps.
 
We give two  applications. First, this result gives a concrete description of the Lawson homology of a complex projective variety in terms of point-wise addition of spherical families of effective algebraic cycles.
Second, we apply this result to monoids built from the unitary, or general linear, representation spaces of discrete groups, leading to results about lifting continuous families of characters to continuous families of representations.
 \end{abstract}

\maketitle

\tableofcontents

\section{Introduction}
For each topological monoid $M$, there is a well-known isomorphism of monoids 
\begin{equation}\label{pi0}\pi_0 (\Omega BM)\isom \Gr (\pi_0 M),\end{equation}
where the left-hand side has the monoid structure induced by loop concatenation, and the right hand side is the Grothendieck group (that is, the ordinary group completion) of the monoid $\pi_0 M$.
%\footnote{A short proof of this result is given in Appendix~\ref{Appendix}.}   
Our main result extends this isomorphism to the higher homotopy groups of $\Omega BM$, under a strong homotopy commutativity condition on $M$. 

In order to state the result, we fix some terminology and notation. We will write $[X,Y]$ to denote the set of \e{unbased} homotopy classes of maps between spaces $X$ and $Y$. Note that if $M$ is a topological monoid, then $[X, M]$ inherits the structure of a monoid, and if $M$ is homotopy commutative then $[X, M]$ is abelian. Given a (discrete) abelian monoid $A$ and a submonoid $B \leqs A$, the quotient monoid is defined by $A/B := A/{\sim}$, where $a\sim a'$ if and only if there exist $b, b'\in B$ such that $a+b = a'+b'$. The 
natural surjective monoid homomorphism 
 $A\to A/B$ has the expected universal property, making it the cokernel of the inclusion $B\injects A$ (in the category of monoids).

\begin{theorem}  $\label{gen-thm-intro}$ Let $M$ be a 
 topological monoid such that for each $x\in M$, there exists a homotopy between the maps $(m,n) \goesto mn$ and $(m,n)\goesto nm$ that is constant on $(x,x)$.  Then for each $k > 0$, there is a natural isomorphism
\begin{equation} \label{Gamma-intro} [S^k, M]/\pi_0 M \srm{\isom} \pi_k (\Omega BM),\end{equation}
where $\pi_0 M \leqs[S^k, M]$ is the  submonoid of nullhomotopic maps $S^k\to M$.
\end{theorem}

We provide an explicit description of the isomorphism (\ref{Gamma-intro}) in Section~\ref{Gamma}.
Our proof of Theorem~\ref{gen-thm-intro} utilizes the McDuff--Segal approach to group completion~\cite{McDuff-Segal}.
The argument in fact requires only a somewhat weaker homotopy commutativity condition on $M$;
see Theorem~\ref{gen-thm} and Proposition~\ref{univ}.
Note that this result \e{does not} hold when $k=0$, because $[S^0, M]/\pi_0 M \isom \pi_0 M$, whereas $\pi_0 (\Omega BM)$ is the group completion of $\pi_0 M$.

To  avoid point-set hypotheses on the monoid $M$,  we form the classifying $BM$ using the ``thick" geometric realization of the simplicial bar construction on $M$,
and we work in the category of compactly generated spaces. In particular, when  forming the bar construction we use the compactly generated topology on the product $M^n$ before taking geometric realization.\footnote{Note that if $M$ is a topological monoid that is not compactly generated, and $k(M)$ denotes $M$ with the compactly generated topology, then $k(M)$ is also a topological monoid under the same multiplication map. 
Moreover, the left-hand side of the isomorphism (\ref{Gamma-intro}) is unaffected by replacing $M$ by $k(M)$ (more generally, for any space $X$ the identity map $k(X)\to X$ is always a weak equivalence, since a function from a compact Hausdorff space to $X$ is continuous if and only if it is continuous as a map to $k(X)$).}
 In the examples considered in this paper, the inclusion of identity element into $M$ is 
 a closed cofibration (in fact, in our examples the identity element is
 an isolated point), which implies in particular that the degeneracy maps in the simplicial bar construction are all closed cofibrations.   This in turn implies that the natural map from $BM$   to the ``thin"  realization of the bar construction is  homotopy equivalence (this is originally due to Segal~\cite{Segal-cat-coh}; see also Wang~\cite{Wang-realization}).

Another formulation of the conclusion of Theorem~\ref{gen-thm-intro} 
is that $\pi_k (\Omega BM)$ is isomorphic to the quotient of Grothendieck groups
\begin{equation}\label{Gr}\dfrac{\Gr[S^k, M]}{\Gr (\pi_0 M)}.\end{equation}
In this formulation, we recover the isomorphism (\ref{pi0}) when $k=0$.\footnote{We note that that homotopy commutativity of $M$ is important here: for example, if $M=G$ is a non-abelian group, then $\Gr[S^0, G] \isom G\cross G$, and the subgroup of constant maps is the diagonal, which is \e{not} normal.}
We will see that for $k>0$, every element of $\pi_k (\Omega BM)$ can be represented by some spherical family $S^k \to M$, because
each element in $[S^k, M]$ is invertible modulo $\pi_0 M$. This allows us to state Theorem~\ref{gen-thm-intro} without reference to Grothendieck groups.

The proof of Theorem~\ref{gen-thm-intro} starts with the case where $\pi_0 M$ contains a cofinal cyclic submonoid generated by some $m_0\in M$. Then $\Gr  (\pi_0 M)$ is the colimit of the sequence $\pi_0 M \xmaps{\bullet m_0} \pi_0 M \xmaps{\bullet m_0} \cdots$, where $\bullet m_0$ denotes multiplication by $m_0$. In this case, Proposition~\ref{SGL-prop} shows that the corresponding infinite mapping telescope is a model for $\Omega BM$, a fact which may be of independent interest.

It is important to note that the isomorphism in Theorem~\ref{gen-thm-intro} \e{does not} hold for all homotopy commutative topological monoids, 
and in fact  the situation we are considering in this paper is in a sense orthogonal to many of the standard applications of group completion, in which Quillen's plus construction plays a critical role.
  For instance, let  
$$M = \coprod_P B \mathrm{Aut} (P),$$
where $R$ is a ring and 
  $P$ runs over a set of representatives for the isomorphism classes of finitely generated projective $R$--modules.  Direct sum makes $M$ a homotopy commutative topological monoid, and the homotopy groups $\pi_* (\Omega BM)\isom K_* (R)$ are the algebraic $K$--theory groups of $R$~\cite{McDuff-Segal}.  However,
$B \mathrm{Aut} (P)$ is the classifying space of the discrete group $\mathrm{Aut} (P)$, so for
$k\geqs 2$ we have $\pi_k (B \mathrm{Aut} (P)) = 0$ and hence every map $S^k\to M$ is nullhomotopic.  Thus the groups (\ref{Gr}) are zero for $k\geqs 2$, whereas $K_* (R)$ is in general quite complicated.

We consider two main applications of Theorem~\ref{gen-thm-intro}: to representation spaces of discrete groups and to Lawson homology of complex projective varieties. 
Given a discrete group $G$, consider the monoids
$$\Rep(G) = \coprod_{n=0}^\infty \Hom(G, \U(n)) \textrm{\,\,\,\, and \,\,\,\,}  \Rep(G)_{\hU} = \coprod_{n=0}^\infty \Hom(G, \U(n))_{\hU(n)}$$
formed from the  unitary representation spaces of $G$
and from their homotopy orbit spaces (Borel constructions)  with respect to the conjugation action of $U(n)$.  (The monoid operations are induced by block sum of matrices.) The homotopy groups $\pi_m \Omega B (\Rep(G)_{\hU})$ are the deformation $K$--theory groups of $G$ introduced by Carlsson and studied by T. Lawson~\cite{Lawson-prod, Lawson-simul} and the author~\cite{Ramras-excision, Ramras-crystal, Ramras-stable-moduli, Ramras-TAS}. Using Theorem~\ref{gen-thm-intro}, we show that for $m>0$
$$ \K_m (G) := \pi_m \left( \Omega B \Rep(G)_{\hU} \right) \isom  \left([S^m, \Rep(G)]/\pi_0 \Rep (G) \right) \oplus K^{-m} (*),$$
where $K^{-m} (*)$ is the complex $K$--theory of a point (see  Theorem~\ref{Kdef}, where a general linear version is also established).
This lends some additional justification to the term \e{deformation} $K$--\e{theory}, by showing that these groups parametrize continuous spherical families (``deformations") of representations. The results in this article are used in the author's preprint~\cite{Ramras-TAS} to give an explicit description of the \e{coassembly map} linking deformation $K$--theory to topological $K$--theory, and to study spaces of \e{flat connections} over the 3--dimensional Heisenberg manifold.

One can also apply Theorem~\ref{gen-thm-intro} to the abelian monoid 
$$\ol{\Rep}(G) = \coprod_{n=0}^\infty \Hom(G, \U(n))/U(n)$$ 
of \e{character varieties}; see Section~\ref{Rdef-sec}.
In Section~\ref{char-sec}, we use this approach to study the problem of lifting a (spherical) family of characters to a family of representations. For instance, in Proposition~\ref{prop:surf} we show that when $G$ is the fundamental group of a compact, aspherical surface, every map from a sphere into $\Hom(G, \U(n))/U(n)$ can be lifted, up to homotopy, to a map into $\Hom(G, \U(N))$ for some $N\geqs n$.

In the case of Lawson homology, our results give a direct description of these groups in terms of the spherical families of algebraic cycles (see Section~\ref{Lawson-sec}), lending concrete meaning to Friedlander's
suggestion that the Lawson homology of a complex projective variety $X$ ``can be naively viewed as the group of homotopy classes of $S^i$--parameterized families of $r$--dimensional algebraic cycles on $X$"~\cite[p. 55]{Friedlander-91}. 

In the Appendices, we give a shorter proof of Theorem~\ref{gen-thm-intro}  in the abelian case (using Quillen's approach to group completion~\cite[Appendix Q]{Friedlander-Mazur} instead of the McDuff--Segal approach) and we prove a

 \vspace{.1in} 
 
\section*{Notation and conventions}   

Throughout the paper, we work in the category of compactly generated spaces as in~\cite{ERW, Lewis-thesis} and 
%(except in Section~\ref{LH-sec}) 
we use the  \e{thick} geometric realization $||-||$ of simplicial spaces (and sets).

For based spaces $(X, x_0)$ and $(Y, y_0)$, 
the set of based homotopy classes of based maps will be  denoted by $\la (X, x_0), (Y, y_0) \ra = \langle X, Y \rangle$, while 
the set of unbased homotopy classes of unbased maps $X\to Y$ will be denoted  $[X, Y]$.

The constant map $X\to \{y\}$ will  be written $c_y$, or sometimes simply $y$, when $X$ is clear from context.
The path component of $x\in X$ is denoted by $[x]$, and we write $x\heq x'$ to indicate that $[x] = [x']$.  
The based homotopy class of a based map $\phi\co (X,x_0) \to (Y, y_0)$  is denoted by $\langle \phi \rangle$, and  its unbased homotopy class by $[\phi]$.  We write $\phi \heq \phi'$ to indicate that $[\phi] = [\phi']$. 

For $m\geqs 1$, we view $\pi_m (X, x_0)$ as the group $\la (S^m, 1), (X, x_0)\ra$, with multiplication $\cc$ defined via concatenation in the first coordinate of 
$$(I^m/\partial I^m, [\partial I^m]) \homeo (S^m, 1).$$
For $\phi \co (S^m, 1) \to (X, x_0)$ with $m\geqs 1$,  we let 
$\overline{\phi}$ denote the reverse of $\phi$ with respect to the first coordinate of $I^m$ (so that $\langle \overline{\phi}\rangle   = 
\langle  \phi  \rangle^{-1}$ in $\pi_m (X, x_0)$).   
 
\section*{Acknowledgements}\label{ackref}
The results in this article had a long genesis, beginning with conversations among the author, Willett, and Yu in 2010 surrounding the work in~\cite{RWY}.
Weaker versions of the main result were included in a  draft of the author's joint paper with Tom Baird~\cite{Baird-Ramras-arxiv-v3}, with different arguments. Another version of these results, closer to the present version, appeared in the author's preprint~\cite{Ramras-TAS}.  
The author thanks Baird, Willett, and Yu, as well as  Peter May, for helpful discussions. Comments from the referees helped improve the exposition.

\section{The homotopy group completion}\label{htpy-sec}

In this section we formulate  our main result (Theorem~\ref{gen-thm}) describing the homotopy groups of the homotopy group completion for certain topological monoids.  This result applies to all (strictly) abelian monoids, and to the monoids of representations defined in the Introduction (see Section~\ref{Kdef-sec}).  
 
We begin by describing the general context.
Throughout, $M$ denotes a topological monoid, with monoid operation $(m,n)\goesto m \bullet n$
 (so $\bullet$ is continuous, associative, and there exists a strict identity element in $e\in M$).  We define
 $$m^k = \underbrace{m \bt \cdots \bt m}_k,$$ 
and by convention  $m^0 = e$. 
All of our results will require $M$ to be homotopy commutative, in the sense that there exists a (possibly unbased) homotopy between $\bt$ and $\bt\circ \tau$, where 
$$\tau \co M\cross M\to M\cross M$$
is the twist map $\tau (m,n) = (n,m)$.
The  \e{classifying space} of $M$, denoted $BM$, is the (thick) geometric realization of the topological category $\underline{M}$ with one object $*$ and with morphism space $M$.  Composition in $\underline{M}$ is given by $m\circ n = m\bullet n$.  The space $BM$ has a natural basepoint $*$ corresponding to the unique object in $\underline{M}$.  
We note that the nerve $N. \ul{M}$, which is the simplicial space underlying $BM$, is the \e{simplicial bar construction} on $M$, and has the form $[n]\goesto M^n$.
The  \e{homotopy group completion} of $M$ is the based loop space $\Omega BM$.  

There is a natural map 
$$\gamma \co M\to \Omega BM,$$
adjoint to the natural map $S^1 \sm M \to BM$ (see Section~\ref{Gamma} for further discussion of this map), and it is a standard fact that  
$\gamma$ induces an isomorphism 
\begin{equation}\label{Gr2} \Gr (\pi_0 M) \srm{\isom} \pi_0 ( \Omega BM),\end{equation}
 where $\Gr$ denotes the Grothendieck group.

  In this section, we present conditions on $M$  under which the higher homotopy groups of $\Omega BM$ can be described in a manner analogous to (\ref{Gr2}).
 Given a space $X$, the space $\Map(X, M)$ is a topological monoid, whose monoid operation  we denote simply by $\bullet$.  
If $M$ is homotopy commutative, which we assume from here on, then $[X,M]$ is an abelian monoid and $\Gr [X,M]$ is an abelian group, so we will use additive notation $+$ and $-$.

Choosing a basepoint $x\in X$ gives a monoid homomorphism
$$[X, M] \maps \pi_0 (M)$$
defined by restriction to $x$.  (We will usually view $\pi_0$ as the functor taking a space to its set of path components. For a based space $(X, x_0)$, we will sometimes use the isomorphic functor $X\goesto \pi_0 (X, x_0) = \la S^0, X\ra$.) This homomorphism is split by the homomorphism
$$\pi_0 (M) \to [X,M]$$
sending each path component to the homotopy class of a constant map into that component.  Hence $[X, M]$ contains a copy of $\pi_0 (M)$ (consisting of nullhomotopic maps), and it follows that $\Gr[X, M]$ contains a copy of $\Gr(\pi_0 M)$ as a direct summand (consisting of formal differences between nullhomotopic maps).

\begin{definition}$\label{a-d}$ For $k\geqs 0$, let $\Pi_k (M)$ denote the Grothendieck group  $\Gr [S^k , M]$, and define
$$\rh_k (M) = \Pi_k (M)/\Gr(\pi_0 M).$$
\end{definition}

Note that we have
  natural direct sum decompositions 
$$\Pi_k (M) \isom \rh_k (M) \oplus \Gr(\pi_0 M),$$ 
with $\rh_k (M)$ corresponding to the subgroup consisting of those formal differences $[\phi] - [\psi]$ for which $\phi(1)$ and $\psi(1)$ lie in the same path component of $M$.  Note also that if $\Gr(\pi_0 M)$ is abelian, then there is a natural isomorphism 
\begin{equation}\label{eq:k=0}\rh_0 (M) \isom \Gr(\pi_0 M),\end{equation}
induced by sending the class represented by $f\co S^0 = \{\pm 1\} \to M$ to $[f(-1)] - [f(1)]$.

In Section~\ref{Gamma}, we will construct a natural map 
$$\wt{\Gamma} \co   \rh_k (M) \maps \pi_k (\Omega BM),$$
and we will show that $\wt{\Gamma}$ is an isomorphism under certain conditions on $M$.  We now explain these conditions.

We will need to consider the action of $\pi_1 (M, m)$ on $\pi_k (M, m)$ ($k\geqs 1$).  We use the conventions in Hatcher~\cite[Section 4.1]{Hatcher}, so that this is a \e{left} action.  We note that if $[m]$ is invertible in $\pi_0 (M)$, then this action is trivial (this follows from~\cite[Example 4A.3]{Hatcher}, for instance, which shows that the identity component of an $H$--space is always simple).

\begin{definition}\label{a-def} Consider a topological monoid $M$ and a natural number $k\geqs 1$.  Given $m\in M$, let $[m]\subset M$ be its path component, viewed as a subspace of $M$.
We say that  $m$ is   $k$--\e{anchored} if there exists a homotopy 
$$H \co [m]\cross [m]\cross I\to [m^2]$$
 such that $H_0 = \bt$, $H_1 = \bt \circ \tau$, 
and the loop $\eta(t) = H (m, m, t)$ acts trivially on $\pi_k (M, m^{2})$.  When $k$ is clear from context, we will refer to $H$ as a \e{homotopy anchoring} $m$.  We say that $m$ is \e{strongly} $k$--anchored if there are infinitely many $n\in \bbN$ for which 
$m^n$ is $k$--anchored.

We say that $m$ is (strongly) \e{anchored} if it is (strongly) $k$--anchored for all $k\geqs 1$.
We say that a path component $C$ of $M$ is (strongly) $k$--anchored (or anchored) if there exists an element 
  $m\in C$ that is (strongly) $k$--anchored (respectively, anchored).
  
To simplify some subsequent statements, all $m\in M$ are considered to be (strongly) $0$--anchored.
\end{definition}

\begin{remark} It is an elementary exercise to check that if $m_0\heq m_1$ in $M$, then $m_0$ is (strongly) $k$--anchored if and only if $m_1$ is (strongly) $k$--anchored.
\end{remark}

\begin{examples}  If $M$ is (strictly) abelian, then every element of $M$ is  strongly anchored, since we can take $H$ to be the constant homotopy.

If $M$ is homotopy commutative, then every path component of $M$ with abelian fundamental group  is 1--anchored (since the action of $\pi_1$ on itself is conjugation). If $M$ is homotopy commutative and every path component of $M$ is a simple space (e.g. if $\pi_0 (M)$ is a group, or if all components are simply connected), then every element in $M$ is strongly anchored.
\end{examples}

We will see more interesting examples in Section~\ref{Kdef-sec}.

\begin{remark}$\label{wa-rmk}$ In \cite{Ramras-excision}, an element $m_0\in M$ is called  anchored  if all powers of $m_0$ are anchored and the loops $\eta$ described in Definition~\ref{a-def} are all \e{constant}.
Lemma~\ref{simple} below provides a strengthening of \cite[Corollary 3.14]{Ramras-excision}, and shows that the results in that article hold if one simply requires   $m_0$ to be strongly $1$--anchored in the sense defined above $($see also \cite[Remark 3.7]{Ramras-excision}$)$.  
\end{remark}

In order to motivate the construction of the map $\wt{\Gamma}$ in Section~\ref{Gamma}, we now formulate the strongest version of  our main result.

Recall that a subset $S$ of a monoid $N$ is called \e{cofinal} if for each $n\in N$ there exists $n'\in N$ such that $n\bt n' \in S$.

\begin{theorem}  $\label{gen-thm}$ Let $M$ be a  
homotopy commutative  
 topological monoid such that the subset of strongly 1--anchored components 
 and the subset of strongly $k$--anchored components 
are both cofinal in $\pi_0 (M)$ (for some $k\geqs 0$). Then 
the natural map
$$\wt{\Gamma}\co \rh_k (M) \maps \pi_k (\Omega BM)$$
is an isomorphism.
\end{theorem}

We note that for $k=0$, the assumption on 1--anchored components in Theorem~\ref{gen-thm} is unnecessary (and the proof given below does not use this assumption when $k=0$). In fact, %as discussed in Appendix~\ref{Appendix}, 
for every topological monoid $M$ the natural map $M\to \Omega BM$ induces an isomorphism $\Gr (\pi_0(M)) \isom \pi_0 (\Omega BM$),
and it follows from (\ref{eq:k=0}) that the map 
$$\wt{\Gamma}\co \rh_0 (M) \maps \pi_0 (\Omega BM)$$
is an isomorphism so long as $\Gr(\pi_0 M)$ is abelian. 

\

We end this section by establishing a helpful universal property of the natural map
\begin{equation}\label{pi}[S^k, M] \srm{i} \Gr[S^k, M] \srm{q} \rh_k (M),\end{equation}
where $i$ is the universal map from the monoid $[S^k, M]$ to its group completion, and $q$ is the quotient map.  We denote the composite (\ref{pi}) by $\pi$.

\begin{definition} Let $(A, +)$ be a (discrete) abelian monoid, and let $A'\leqs A$ be a submonoid. 
Define an equivalence relation $\sim$ on $A$ by $a_1 \sim a_2$ if and only if there exist $a'_1, a'_2 \in A'$ such that $a_1+a'_1 = a_2 +a'_2$.
The quotient monoid $A/A'$ is defined to be the quotient set $A/\!\sim$ with operation  $[a_1] + [a_2] = [a_1+a_2]$.
\end{definition}

The fact that $\sim$ is an equivalence relation follows from the assumption that $A$ is abelian, and it is immediate that $(A/A', +)$ is a monoid. Furthermore, the quotient map $\pi\co A\to A/A'$ is a monoid map, with the expected universal property: if $P$ is a (not necessarily abelian) monoid with identity element $0\in P$, and $f\co A\to P$ is a map of monoids satisfying $f(A') = 0$, then there exists a unique monoid map 
$\ol{f}\co A/A'\to P$ satisfying $\ol{f} \circ \pi = f$:
$$\xymatrix{ A' \ar[r] \ar@/^1.5pc/[rrr]^0 & A \ar[rr]^f \ar@{->>}[dr]^{\pi} & & P.\\
& & A/A' \ar@{-->}[ur]^{\exists !}}$$

\begin{proposition}$\label{univ}$ Fix $k\geqs 1$, and
let $M$ be a homotopy commutative topological monoid in which the subset of $k$--anchored components is cofinal.  Then 
  $$\pi \co [S^k, M] \to \rh_k (M)$$
is surjective, and in fact 
satisfies the universal property of the quotient map for the submonoid $\pi_0 M \leqs [S^k, M]$.
Consequently,
$$\rh_k (M)\isom [S^k, M]/\pi_0 M,$$
and
$\pi ([\phi]) = \pi ([\psi])$ if and only if there exist constant maps   $c, d\co S^k\to  M$ 
such that $\phi \bt c \heq \psi \bt d$.
\end{proposition}

For the proof of Proposition~\ref{univ}, we need a version of the Eckmann--Hilton argument, and first we record a basic fact regarding the action of the fundamental group on higher homotopy.

\begin{lemma}\label{act}  Consider a $($not necessarily based$)$ homotopy $\alpha_s$ of maps 
$$(S^k, 1) \to (X, x_0)$$ 
$(k\geqs 1)$, and let 
$\eta (t) = \alpha_t (1)$ be the track of this homotopy on the basepoint $1\in S^k$.  Then $\la \alpha_0 \ra = \la \eta \ra  \cdot \la \alpha_1\ra$ in $\pi_k (X, \alpha_0 (1))$.
\end{lemma}
\begin{proof} In general, the action of $\pi_1 (X, x_0)$ on $\pi_k (X, x_0)$ is induced by an operation which takes in a map $\gamma\co [0,s]\to X$ (for some $s\in [0,1]$)
 and a map $\alpha \co (I^k, \partial I^k) \to (X, \gamma(s))$ and produces a map 
$$\gamma \cdot \alpha\co (I^k, \partial I^k) \to (X, \gamma(0))$$ 
defined by shrinking the domain of $\alpha$ to a concentric cube  $C\subset I^k$ of side length $1-s/3$ and filling in the path $\gamma$ on each radial segment connecting $\partial C$ to $\partial I^k$ (compare with Hatcher~\cite[Section 4.1]{Hatcher}, for instance).
In this language, the desired homotopy is simply
$s\goesto \eta|_{[0,s]} \cdot \alpha_s$.
\end{proof}

\begin{lemma}$\label{EH}$
Fix $k\geqs 1$ and 
let $M$ be a  topological monoid and let $m\in M$ be $k$--anchored.
Then for any $\phi, \psi\co S^k \to M$ with $\phi(1) = \psi(1) = m$ we have a based homotopy
$$ \phi   \bullet \psi   \heq   ( \phi \bullet  c_m   )\cc ( \psi \bt c_m  ) = (\phi \cc \psi) \bt c_m.$$
In  particular, setting $\psi = \ol{\phi}$, we obtain based homotopies
$$\phi \bt \overline{\phi}  \heq \overline{\phi}   \bt   \phi   \heq  c_{m^{2}}.$$
\end{lemma}

\begin{proof}  Just as in the ordinary Eckmann--Hilton argument, the point is that $\bt$ is a homomorphism 
$$\pi_k (M, m) \cross \pi_k (M, m) \srm{\bt} \pi_k (M, m^2).$$
The relevant equation holds on the nose, not just up to homotopy: for all maps $\alpha, \beta, \alpha', \beta'\co S^k \to M$ satisfying
$\alpha(1) = \beta(1)$ and $\alpha'(1) = \beta'(1)$,
$$(\alpha\cc \beta) \bt (\alpha' \cc \beta') = (\alpha \bt \alpha') \cc (\beta \bt \beta').$$
Hence we have:
$$\phi \bt \psi \heq (\phi \cc c_m) \bt (c_m \cc \psi) = (\phi \bt c_m) \cc (c_m \bt \psi).
$$

To complete the proof, it suffices to show that $\la c_m \bt \psi \ra = \la \psi \bt c_m \ra$.
Let $H$ be a homotopy anchoring $m$, and set $\eta (t) = H(m, m, t)$.
By Lemma~\ref{act},   
$$\la c_m \bt \psi \ra  = \la \eta \ra \cdot \la \psi  \bt  c_m\ra,$$
and  since $H$ anchors $m$, we have $ \la \eta \ra \cdot \la \psi  \bt  c_m\ra  =    \la  \psi  \bt  c_m\ra$.\end{proof}

\begin{proof}[Proof of Proposition~\ref{univ}] 
Each element in $\Gr[S^k, M]$ has the form $[\phi]-[\psi]$ for some $\phi, \psi\co S^k \to M$.  By assumption, there exists $m\in M$ such that $\psi(1) \bt m$ is $k$--anchored.  Adding $[c_m]$ to both $[\phi]$ and $[\psi]$ if necessary, we may assume that $\psi(1)$ is $k$--anchored. 

By Lemma~\ref{EH}, $\psi\bt\ol{\psi}$ is nullhomotopic, so the element
$$[\phi]-[\psi] = [\phi\bt\ol{\psi}] - [\psi\bt\ol{\psi}] \in \Gr [S^k, M]$$
 is equivalent, modulo $\Gr (\pi_0 M)$, to $[\phi\bt\ol{\psi}]$, which is in the image of $\pi$.  Hence $\pi$ is surjective.

Now say $f \co [S^k, M] \to P$ is a 
homomorphism sending all nullhomotopic maps to the identity.  Since $\pi$ is surjective,  
$\ol{f}$ is completely determined by the equation $\ol{f} (\pi [\phi]) = f([\phi])$.  
To prove that $\ol{f}$ is well-defined, say
$\pi ([\phi]) = \pi ([\psi])$ for some $\phi, \psi\co S^k \to M$.
Then there exist $x, y\in M$ such that
$$[\phi] - [\psi] = [c_x] - [c_y],$$
in $\Gr[S^k, M]$, and hence there exists  $\tau\co S^k \to M$ such that
$$[\phi] + [c_y] + [\tau] = [c_x] + [\psi] + [\tau]$$
in $[S^k, M]$.  Again, we may assume without loss of generality that $\tau(1)$ is $k$--anchored.
 Adding $[\ol{\tau}]$ to both sides and applying $f$, we have
\begin{equation}\label{homom}
f([\phi]) + f([c_y]) + f([\tau\bt\ol{\tau}]) = f([c_x]) + f([\psi]) + f([\tau\bt\ol{\tau}]).
\end{equation}
By Lemma~\ref{EH}, $\tau\bt\ol{\tau}$ is nullhomotopic.  Since $f$ sends all nullhomotopic maps to the identity, Equation (\ref{homom}) reduces to $f([\phi]) = f([\psi])$, showing that $\ol{f}$ is well-defined.  It follows  from the equation $f = \ol{f} \circ \pi$ (together with surjectivity of $\pi$) that $\ol{f}$ is a homomorphism as well.
 \end{proof}

\section{Construction of the map $\wt{\Gamma}$}\label{Gamma}

We now give the details behind the construction of the natural map
$$\wt{\Gamma}\co \rh_k (M) = \Gr [S^k , M]/\Gr(\pi_0 M) \maps \pi_k (\Omega BM).$$

We begin by discussing an alternate model for the groups $\pi_k (\Omega BM)$.
Sending a based homotopy class to its unbased homotopy class defines a natural map
\begin{equation}\label{cn} J \co \pi_k \Omega BM = \la (S^k, 1), (\Omega BM, c_*) \ra \maps [S^k, \Omega BM]_0,\end{equation}
 where the right-hand side is the set of unbased homotopy classes of maps $f\co S^k\to \Omega BM$ such that $f(1)$ is homotopic  to the constant loop $c_*$. 
When $k=0$, it is immediate that $J$ is bijective.
Concatenation of loops makes $\Omega BM$ into an $H$--space with the constant loop $c_*$ as identity, so the identity component $\Omega_0 BM$ of $\Omega BM$ is simple. 
It now follows from Lemma~\ref{act} (see also \cite[Section 4.A]{Hatcher}) that $J$ is an isomorphism for $k \geqs 1$ as well; note also that in this case we have
$$[S^k, \Omega BM]_0 = [S^k, \Omega_0 BM].$$

There is a natural operation on $[S^k, \Omega BM]_0$  coming from the $H$--space structure of $\Omega BM$.  This operation is induced by the point-wise concatenation map
\begin{equation}\label{bd}\boxdot \, \co \Map(S^k, \Omega BM) \cross \Map(S^k, \Omega BM) \maps \Map(S^k, \Omega BM),
\end{equation}
defined by 
$$(\alpha \boxdot \beta) (z) =  \alpha(z) \cc \beta (z).$$
The map $\boxdot$ also induces an operation on $\la S^k, \Omega BM\ra = \pi_k (\Omega BM)$, and the bijection (\ref{cn}) is a homomorphism with respect to these operations.
Moreover, when $k\geqs 1$, the Eckmann-Hilton  argument~\cite{Eckmann-Hilton} shows that the operation on $\la S^k, \Omega BM\ra$ induced by $\boxdot$ agrees with the usual multiplication operation on $\pi_k (\Omega BM)$.  
With this understood, we may now view $J$ as a group isomorphism.

Recall that there is a natural map
\begin{equation}\label{eta} [0,1]\cross M\to BM\end{equation}
resulting from the fact that the simplicial space underlying $BM$ has $M$ as its space of 1-simplices\footnote{Our conventions on geometric realization come from Milnor~\cite{Milnor-realization}, and $(\ref{eta})$ is induced by the homeomorphism $I = [0,1] \to \Delta^1 = \{(t_0, t_1, t_2) \in \bbR^3\,|\, 0=t_0\leqs t_1 \leqs t_2 = 1\}$ given by $t\goesto (0, t, 1)$.}. 
The adjoint of  (\ref{eta}) is a map from $M$ to $ \Map ([0,1], BM)$,
and since $BM$ has a single $0$--simplex, 
this adjoint in fact lands in the based loop space of $BM$. We denote this map by
\begin{equation} \label{gamma} \gamma \co M\maps \Omega BM.\end{equation}
Note that $\gamma$ is natural with respect to continuous homomorphisms of topological monoids.
One might like to define a map 
$$[S^k , M] \maps [S^k, \Omega BM]_0 \isom \pi_k (\Omega BM)$$
 via composition with $\gamma$, but some correction is needed to make this map land in $ [S^k, \Omega BM]_0$.

Given $\alpha\in \Omega BM$ and $g\co S^k \to \Omega BM$, we simplify notation by writing $\alpha \boxdot g$ in place of $c_\alpha \boxdot g$.
For each $k\geqs 0$, we now define  
$$\Gamma \co [S^k , M] \maps [S^k, \Omega BM]$$
$$\hspace{.6in} [f]\goesto [\ol{\gamma(f(1))} \boxdot (\gamma \circ f)].$$ 
We note that there is a potential ambiguity in this notation: the symbol $\ol{\gamma(f(1))}$ refers to the constant map with image $\ol{\gamma (f(1))}\in \Omega BM$, not to the reverse of the constant map with image $\gamma(f(1))$ (of course a constant map is its own reverse).  We will continue to use this notation throughout the section.

It is straightforward to check that $\Gamma$ is well-defined on unbased homotopy classes, and $\Gamma ([f]) \in [S^k, \Omega_0 BM]$ for every $[f]$, since evaluating at $1\in S^k$ gives the loop $\ol{\gamma( f(1))} \cc \gamma (f(1)) \heq c_*$.

\begin{proposition}$\label{Gamma-hom}$ For each $k\geqs 0$, the function $\Gamma$ is a monoid homomorphism, natural in $M$. Moreover $\Gamma$ induces a natural homomorphism
$$\ol{\Gamma}_M \co \rh_k (M) \maps [S^k, \Omega BM]_0.$$
\end{proposition}

We can now define the desired map $\wt{\Gamma}\co \rh_k (M) \maps \pi_k (\Omega BM)$.

\begin{definition}\label{wtGamma} Let $M$ be a homotopy commutative topological monoid.  For $k\geqs 0$, we define
$$\wt{\Gamma} = \wt{\Gamma}_M \defn J^{-1} \circ \ol{\Gamma}_M \co \rh_k (M) \maps \pi_k (\Omega BM),$$
where $J$ is the (natural) isomorphism (\ref{cn}). 
\end{definition}

 The proof of Proposition~\ref{Gamma-hom} will use the following elementary lemma.

\begin{lemma}\label{loops}
Let $(M, \bt)$ be a topological monoid.  Then the diagram
\begin{equation}\label{alpha-diag}
\xymatrix{ M\cross M \ar[r]^-{\bt} \ar[d]^{\gamma\cross\gamma} & M\ar[d]^\gamma\\
	\Omega BM \cross \Omega BM \ar[r]^-\cc &\Omega BM}
\end{equation}
is homotopy commutative.  Moreover, if $M$ is homotopy commutative  then
the maps
$$M\cross M \to \Omega BM$$
given by $(m, n) \goesto \gamma(m) \cc \gamma(n)$ and $(m, n) \goesto  \gamma(n)\cc \gamma(m)$ are homotopic.
\end{lemma}
\begin{proof}    The space of 2-simplices in the simplicial space $BM$ homeomorphic to $M\cross M$, with $(m, n)$ corresponding to the sequence of composable morphisms 
$$*\xmaps{(n, *)} * \xmaps{(m,*)} *.$$  

We  describe the desired homotopy $M\cross M\cross I\to \Omega BM$ by specifying its adjoint, which is induced by a map of the form
$$(M\cross M\cross I) \cross I = M\cross M \cross (I \cross I) \xmaps{\Id_M\cross \Id_M \cross H} M\cross M\cross \Delta^2 \srm{\pi} BM,$$
where $H\co I\cross I \to \Delta^2$ is defined below and $\pi$ is induced by 
the definition of geometric realization.
Set 
$$\Delta^2 = \{( t_1, t_2 )\in I\cross I \,:\,   t_1 \leqs t_2  \},$$
and define 
$$\vec{w}_t =(1-t)(0,1) + t(1/2, 1/2) = (t/2, 1-t/2) \in \Delta^2.$$
The map $H$ is defined by
$$H(t,s) = \begin{cases} 2s \vec{w}_t & \textrm{if} \,\, 0 \leqs s\leqs 1/2\\  (2s-1) (1, 1) + (2-2s) \vec{w}_t & \textrm{if} \,\, 1/2 \leqs s\leqs 1,\end{cases}$$ 
and one may check that it has the desired properties 
(note that we are using the conventions regarding (co)face  maps from~\cite{Milnor-realization}).  This proves commutativity  of 
(\ref{alpha-diag}).

When $M$ is homotopy commutative, the second statement in the lemma follows from the first:  we have
$$\cc \circ (\gamma\cross \gamma) \heq \gamma \circ \bt \heq \gamma \circ \bt \circ \tau \heq \cc \circ (\gamma\cross \gamma) \circ \tau.$$
\end{proof}

\begin{proof}[Proof of Proposition~\ref{Gamma-hom}] First we show that $\Gamma$ is a monoid homomorphism.   
  Given $\phi, \psi \co S^k \to M$, we must show that
$$\Gamma ([\phi \bt \psi])  = \Gamma ([\phi]) \boxdot \Gamma ([\psi]),$$
or in other words that
$$\ol{\gamma (\phi (1) \bt \psi (1))} \boxdot \left(\gamma \circ (\phi \bt \psi)\right)
\heq (\ol{\gamma (\phi (1))} \boxdot \gamma \circ \phi) \boxdot (\ol{\gamma (\psi (1))} \boxdot \gamma \circ \psi).$$
Applying Lemma~\ref{loops} gives
\begin{eqnarray*}
\ol{\gamma (\phi (1) \bt \psi (1))} \boxdot  \left(\gamma \circ (\phi \bt \psi)\right)
& \heq &\left( \ol{  \gamma (\phi (1) ) \cc \gamma(\psi  (1)) }  \right)
			\boxdot \left(\gamma \circ \phi \boxdot \gamma \circ \psi\right)\\
& = & \left( \ol{  \gamma (\psi  (1)) } \cc \ol{  \gamma (\phi  (1) )} \right) \boxdot 
	\left( \gamma \circ \phi \boxdot \gamma \circ \psi \right)\\
\end{eqnarray*}
Since the operation $\boxdot$ is homotopy associative, to complete the proof that $\Gamma$ is a homomorphism it remains only to show that $[\ol{  \gamma(\psi  (1) )}]$, $ [\ol{  \gamma(\phi (1) )}]$, and $[\gamma \circ \phi]$ commute with one another under the operation $\boxdot$.  
By Lemma~\ref{loops}, $[ \gamma(\psi  (1) )]$, $[\gamma(\phi  (1) )]$, and $[\gamma \circ \phi]$ commute with one another, which suffices because $[ \gamma(\psi  (1) )]$ and $[\gamma(\phi  (1) )]$ are the inverses of $[\ol{  \gamma(\psi  (1) )}]$ and $ [\ol{  \gamma(\phi (1) )}]$ (respectively) under $\boxdot$.

It follows from the definitions that $\Gamma$ sends all nullhomotopic maps to the identity element in $[S^k, \Omega BM]_0$. Since $\Gamma$ is a monoid homomorphism, it extends to a group homomorphism out of $\Pi_k (M) = \Gr [S^k, M]$, which must be zero on the subgroup $\Gr (\pi_0 M)$. 
This yields the desired homomorphism
$$\ol{\Gamma}\co \rh_k (M) = \Pi_k (M)/\Gr (\pi_0 M) \maps  [S^k, \Omega BM]_0.$$ 
Naturality of $\Gamma$, and hence of $\ol{\Gamma}$, follows from naturality of $\gamma$. 
\end{proof}

\section{Stably group-like monoids}$\label{sgl-sec}$
We now  show that under certain conditions, it is possible to construct an inverse to the homomorphism
$$\wt{\Gamma}\co \rh_k (M) \maps \pi_k (\Omega BM)$$
introduced in Definition~\ref{wtGamma}.
We recall some terminology from Ramras~\cite[Section 3]{Ramras-excision}.

\begin{definition}$\label{sg}$ A topological monoid $M$ is \e{stably group-like} with respect to an element $[m]\in \pi_0 (M)$ if the submonoid of $\pi_0 (M)$ generated by $[m]$ is cofinal.  More explicitly, $M$  is stably group-like with respect to $m$ if for every $x\in M$, there exists $x'\in M$ and $k\geqs 0$ such that $x\bt x'$ lies in the same path component as $m^k$.

Given $m_0 \in M$, we write $M\xmaps{\bt m_0} M$ to denote the map $m\goesto m\bt m_0$.
We define
$$M_\infty (m_0) = \tele \left(M\xmaps{\bt m_0} M \xmaps{\bt m_0} M\xmaps{\bt m_0} \cdots\right),$$
where the right-hand side is the infinite mapping telescope of this sequence.
As in~\cite{Ramras-excision}, we write points in $M_\infty (m_0)$ as equivalence classes of triples $(m,n,t)$, where $m\in M$, $n\in \bbN$, $t\in [0,1]$, and 
$$(m,n,1) \sim (m\bt m_0, n+1, 0)$$ 
for each $n\in \bbN$.  We always use $e:= [(e, 0, 0)]\in M_\infty (m_0)$ as the basepoint.
\end{definition}

Our next goal is to prove the following special case of Theorem~\ref{gen-thm}, from which the full result will be deduced in Section~\ref{sec:proof} using a colimit argument.

\begin{proposition}$\label{sgl}$ Let $M$ be a proper, homotopy commutative topological monoid that  is stably group-like with respect to a strongly $1$--anchored element $m_0\in M$. If the submonoid of strongly $k$--anchored components is cofinal in $\pi_0 (M)$ for some $k\geqs 0$, then 
$$\wt{\Gamma}\co \rh_k (M) \maps \pi_k (\Omega BM)$$
is an isomorphism.
\end{proposition}

Before giving the proof, we need to review some facts surrounding the Group Completion Theorem of McDuff and Segal~\cite{McDuff-Segal}. A detailed account of this result and its proof appears in~\cite{ERW} (see also~\cite{Miller-Palmer, Randal-Williams-scholium}). The Group Completion Theorem provides an isomorphism
\begin{equation}\label{MS} \pi_k (\Omega BM) \srm{\isom} \pi_k (M_\infty (m_0))\end{equation}
under the conditions in Proposition~\ref{sgl}.
 We will give an explicit description of this isomorphism in Lemma~\ref{MS-iso} below.  The proof of Proposition~\ref{sgl} will then proceed by 
 constructing another map
$$\Psi\co \pi_k (M_\infty (m_0)) \maps \rh_k (M)$$
so that the composite
$$\pi_k (\Omega BM) \srm{\isom} \pi_k (M_\infty (m_0))\srm{\Psi} \rh_k (M)$$
is inverse to $\wt{\Gamma}$.

If $M$ is 
%proper and
 stably group-like with respect to a strongly 1--anchored component $[m_0]$, then (\ref{MS}) is induced by a zig-zag of weak equivalences, as we now explain.
The monoid $M$ acts continuously on $M_\infty (m_0)$ via 
$$m \cdot [(m', n, t)] = [(m\bt m', n, t)].$$
An action of $M$ on a space $X$ gives rise to  a category (internal to $\Top$) with object space $X$ and morphism space $M \cross X$; the morphism 
$(m, x)$ has domain $x$ and range $m\cdot x$, and composition is just multiplication in $M$: $(n, m\cdot x) \circ (m, x) = (n\bt m, x)$.  We denote the classifying space of this category by $X_M$.  Since $\{*\}_M \isom BM$ (where $\{*\}$ is the one-point space),  we get a canonical map
$$q\co X_M \to BM$$
induced by the projection $X\to \{*\}$.  When $X = M_\infty (m_0)$, we call this projection map $q(M, m_0)$.

The isomorphism (\ref{MS})
is induced by a zig-zag of weak equivalences of the form
\begin{equation}\label{MS1}\Omega BM \srou{f}{\heq} \hofib (q(M, m_0))
\srlou{g}{\heq} M_\infty (m_0).
\end{equation}
Here $\hofib (q(M, m_0))$ is the homotopy fiber of $q(M, m_0)$ over the basepoint $* \in BM$.  
Points in $\hofib (q(M, m_0))$ are pairs 
$$(z, \beta)\in (M_\infty (m_0))_M\cross \Map(I, BM)$$ 
with $\beta(0) = q(M, m_0) (z)$ and $\beta(1) = *$.  The basepoint of $\hofib (q(M, m_0))$ is the pair $([e, 0, 0], c_{*})$, where $[(e, 0, 0)] \in (M_\infty (m_0))_M$ corresponds to the point $[(e, 0, 0)]$ in the object space $M_\infty (m_0)$ of the category underlying $(M_\infty (m_0))_M$.  

The first map $f$ in (\ref{MS1}) is induced by sending a based loop $$\alpha\co S^1 \to BM$$ to the point $([(e, 0, 0)], \alpha) \in \hofib (q(M, m_0))$.  It is a weak equivalence because $(M_\infty (m_0))_M$ is weakly contractible (see~\cite[p. 281]{McDuff-Segal},~\cite[pp. 2251--2252]{Ramras-excision}, or~\cite[Section 6.4]{ERW}).  Note here that $\Omega BM \isom \hofib(* \to BM)$.

The second map $g$ in (\ref{MS1}) is the natural inclusion of the fiber of $q(M, m_0)$ over $*\in BM$ into the homotopy fiber.  The fact that $g$ is a weak equivalence is established in~\cite[Proof of Theorem 3.6]{Ramras-excision} under a stronger anchoring condition on $m_0$ (see Remark~\ref{wa-rmk}). The proof given in~\cite[Proof of Theorem 3.6]{Ramras-excision} relies on the fact that $g$ induces isomorphisms in homology for all local (abelian) coefficient systems. A full proof of this fact, due to Randal--Williams, appeared only recently~\cite{Randal-Williams-scholium}. In the present context, however, considerations of non-trivial coefficient systems can be avoided: we  show in Lemma~\ref{simple} below  that $M_\infty (m_0)$ is in fact a \e{simple space} whenever $m_0$ is strongly $1$--anchored. Since $\Omega BM$ is a group-like $H$--space, it too is simple, as is $\hofib(q(M,m_0))\heq \Omega BM$. Thus in the present situation, to conclude that $g$ is a weak equivalence
it suffices (by~\cite[Proposition 4.74]{Hatcher}) to show that $g$ induces an isomorphism in \e{integral} homology; this is proven, for instance, in~\cite[Section 6.4]{ERW}.

 For each $k\geqs 0$, there is a natural isomorphism
\begin{equation} \label{infty} \pi_k (M_\infty (m_0), e) \isom \colim \left(\pi_k (M, e) \xmaps{\bt m_0} \pi_k (M, m_0) \xmaps{\bt m_0}   \cdots\right),
\end{equation}
where the maps in the colimit on the right are those induced by (right) multiplication by the constant map $c_{m_0}$.  
We will denote the colimit on the right by 
\begin{equation}\label{colim}\colim \limits_{n\to \infty} \pi_k (M, m_0^n).\end{equation}
It will be helpful to describe the isomorphism (\ref{infty}) explicitly. Each finite stage 
of the mapping telescope $M_\infty (m_0)$ deformation-retracts, in a canonical manner, to its right end $M\cross \{n\}\cross\{0\}$  ($n=0,1,\ldots$). Let 
$$\xi_n\co [0,1]\to M_\infty (m_0)$$
 be the track of this deformation retraction on $[(e,0,0)]$, so that $\xi_n$ is a path from $[(e,0,0)]$ to $[(m_0^n, n, 0)]$.
Then for each $\la \alpha \ra \in \pi_k (M, m_0^n)$, the isomorphism (\ref{infty}) sends $\la \alpha \ra$ to 
$\la \xi_n \cdot \alpha\ra$, where $\cdot$ refers to the operation described in Lemma~\ref{act}.
We refer to $\xi_n \cdot \alpha$ as a representative of $\la \xi_n \cdot \alpha\ra$ at level $n$.
Note that if $x\in \pi_k (M_\infty (m_0), e)$ admits a representative at level $n$, then it also admits representatives at level $N$ for all $N\geqs n$.

\begin{lemma} \label{simple} Let $M$ be a topological monoid that is stably group-like with respect to a strongly $1$--anchored element  $m_0\in M$.
Then $M_\infty (m_0)$ is simple.
\end{lemma}
\begin{proof} 
First we show that $\pi_1 (M_\infty (m_0), e)$ acts trivially on $\pi_k (M_\infty (m_0), e)$.
Given elements in $\pi_1 (M_\infty (m_0), e)$ and $\pi_k (M_\infty (m_0), e)$, we may choose representatives 
$\xi_n \cdot \xi$ and $\xi_n \cdot \alpha$
of these classes (at the same level), and we may assume that there exists a homotopy $H$ anchoring $m_0^{n}$.

Now consider the homotopy 
$$\alpha_t = \begin{cases} \xi_{2n}|_{[3t, 1]} \cdot (\alpha \bt m_0^n), & \textrm{ if } 0 \leqs t \leqs 1/3\\
					\alpha \bt \xi (3(t-1/3)), &  \textrm{ if } 1/3 \leqs t \leqs 2/3\\
					\xi_{2n}|_{[3(1-t), 1]} \cdot (\alpha \bt m_0^n), &  \textrm{ if } 2/3 \leqs t \leqs 1,
		\end{cases}
$$		
(the restricted paths $\xi_{2n}|_{[3t, 1]}$, $ \xi_{2n}|_{[3(t-1), 1]}$ should of course be reparametrized to have domain starting at $0$).
We have $\alpha_0 = \alpha_1 = \xi_{2n} \cdot (\alpha \bt m_0^n)$, and the loop
$\alpha_t (1)$ is precisely $\xi_{2n} \cdot (m_0^n \bt \xi)$. Thus we see that 
$\la \xi_{2n} \cdot (m_0^n \bt \xi)\ra$ acts trivially on $\la \xi_{2n} \cdot (\alpha \bt m_0^n) \ra = \la \xi_{n} \cdot \alpha \ra$. 

We claim that 
\begin{equation}\label{comm-eq}\la  m_0^n \bt  \xi \ra  = \la \xi \bt m_0^n\ra.\end{equation}
It then follows that
$$\la \xi_{2n} \cdot (m_0^n \bt  \xi) \ra  = \la \xi_{2n} \cdot (\xi \bt m_0^n) \ra = \la \xi_n \cdot \xi\ra,$$
allowing us to conclude that $ \la \xi_n \cdot \xi\ra$ acts trivially on $ \la \xi_{n} \cdot \alpha \ra$, as desired.

To prove (\ref{comm-eq},) let $H$ be a homotopy anchoring $m_0^n$, and let $\eta$ be the loop $\eta (t) = H(m_0, m_0, t)$ (so by assumption, $\la \eta\ra$ is central in $\pi_1 (M, m_0^{2n})$).
For each $s\in [0,1]$, let $\eta_s$ be the path $\eta_s (t) = \eta (st)$ ($t\in [0,1]$), and set
$$h_s (t) = H(\xi (t), m_0^n, s),$$
so that $h_s$  is a loop based at $\eta (s) = \eta_s (1)$.
  Let
$$g_s  = \eta_s \cdot h_s\heq \eta_s \cc h_s \cc \ol{\eta}_s.$$
For each $s\in [0,1]$, $g_s$ is   based at $m_0^{2n}$, and $\la g_0\ra = \la\xi \bt m_0^n \ra$.  Since
$$\la g_1 \ra = \la \eta \cdot ( m_0^n \bt  \xi)\ra = \la m_0^n \bt  \xi\ra,$$ 
we see that $g_s$ is a based homotopy
$\xi \bt m_0^n \heq m_0^n \bt  \xi$.

To complete the proof, we must extend this simplicity result to the other path components of $M_\infty (m_0)$. 
The proof is essentially that of~\cite[Corollary 3.14]{Ramras-excision}, so we will be brief. 
Let $C$ be the path component of $M_\infty (m_0)$ 
containing $[(m, n, t)]$. Then there exists $m'\in M$ such that $m' \bt m \heq m_0^N$ for some $N\geqs n$. Consider the maps
$f, g\co M_\infty (m_0) \to M_\infty (m_0)$
defined by
$$f([(x, l, s)]) = [(m' \bt x, l+N-n, s)] \textrm{ and } g(([x,l, s)]) = [(m\bt x, l+n, s)].$$
Then 
 $f(C) \subset C_e$  
 and 
$g(C_e) \subset C$.   
We claim that $g_* \circ f_*$, and hence $f_*$, is injective on homotopy. Simplicity of $C_e$ then implies simplicity of $C$.
Since $$g\circ f \circ ([x,l,0]) = [(m\bt m' \bt x, l+N, 0)]\heq [(x\bt m_0^N, l+N, 0)] = [(x, l, 0)],$$ 
the restriction of $g\circ f$ to $M\cross \{l\}\cross \{0\}$ is homotopic to the inclusion. Injectivity of $g_* \circ f_*$  now follows from the fact that
every compact set in $M_\infty (M)$ can be homotoped into $M\cross \{l\}\cross \{0\}$ for some $l$.
\end{proof}

The preceding discussion is summarized by the following proposition.

\begin{proposition}\label{SGL-prop}  Let $M$ be a %proper, 
homotopy commutative topological monoid that is stably group-like with respect to a strongly $1$--anchored element  $m_0\in M$.
Then there are weak equivalences
\begin{equation}\label{MS1'}\Omega BM \srou{f}{\heq} \hofib (q(M, m_0))
\srlou{g}{\heq} M_\infty (m_0).
\end{equation}
\end{proposition}

Our next goal is to give an explicit description of the isomorphism on homotopy groups provided by (\ref{MS1'}).
Let 
\begin{equation}  \label{i}
\xymatrix@R=2pt{M \ar[r]^-{i_n} & M_\infty (m_0)\\
m \ar@{|->}[r] & [(m, n, 0)]}
\end{equation}
 denote the inclusion of $M$ into the $n$th stage of the mapping telescope, and 
 define 
\begin{equation}  \label{fn}
\xymatrix@R=2pt{M \ar[r]^-{f_n} &  \Omega BM\\
m \ar@{|->}[r] & \ol{\gamma(m_0^n)} \cc \gamma(m),}
\end{equation}
where $\gamma \co M\to \Omega BM$ is the map (\ref{gamma}).

\begin{lemma} \label{MS-iso} Let $M$ and $m_0$ be as in Proposition~\ref{sgl}.  Then for each element $\alpha\in \pi_k (M, m_0^n)$,
the isomorphism
$$\pi_k (\Omega BM) \srm{\isom} \pi_k (M_\infty (m_0))$$
induced by the zig-zag $(\ref{MS1})$ carries $(f_n)_* (\alpha)$ to $(i_n)_* (\alpha)$.  
\end{lemma}

Note that
every class in $\pi_k (M_\infty (m_0))$ has the form $(i_n)_* (\alpha)$ for some $\alpha\in \pi_k (M, m_0^n)$ (by (\ref{infty})), so Lemma~\ref{MS-iso} completely determines the isomorphism 
$$\pi_k (\Omega BM) \srm{\isom} \pi_k (M_\infty (m_0))$$
 induced by (\ref{MS1}).  By abuse of notation, we will denote this isomorphism by $i \circ f^{-1}$ from now on.

\begin{proof}[Proof of Lemma~\ref{MS-iso}]  
It suffices to show that the diagram
\begin{equation} \label{htpy-commutes-diag}
\xymatrix{ &M \ar[dl]_{f_{n}} \ar[dr]^{i_n} \\ \Omega BM \ar[r]^-\heq & \hofib (q(M, m_0)) & M_\infty (m_0), \ar[l]_-\heq}
\end{equation}
is homotopy commutative.
 This can be proven using the argument at the end of the proof of~\cite[Theorem 3.6]{Ramras-excision}.  That argument shows that the diagram commutes after passing to connected components, and it is routine to check that the paths constructed there give rise to a continuous homotopy.  
\end{proof}

\begin{remark}\label{rmk:pi0}
When $k=0$, the colimit $(\ref{colim})$ has a monoid structure defined as follows$:$ 
denoting elements in $\pi_k (M, m_0^n)$ $($$n = 0, 1, 2, \ldots$$)$ by pairs $([m], n)$ with $m\in M$, we define
$$[([m], n)] + [([m'], n')] = [([m\bt m'], n+n')].$$
This monoid structure is in fact a group structure since $M$ is stably group-like with respect to $[m_0]$, and the bijection
$$\pi_0 (\Omega BM) \isom \colim \limits_{n\to \infty} \pi_0 (M, m_0^n)$$
given by composing $(\ref{MS})$ and $(\ref{infty})$
is a monoid isomorphism $($see Ramras~\cite[Theorem 3.6]{Ramras-excision} for details$)$. 
\end{remark}

\begin{proof}[Proof of Proposition~\ref{sgl}]
Given a based homotopy class $\langle \phi\rangle  \in \pi_k (M, m_0^n)$, we define 
$$\Psi_n (\langle \phi\rangle ) = [\phi] \in \rh_k (M).$$
Since all constant maps are trivial in $\rh_k (M)$, the maps $\Psi_n$ are compatible with the structure maps for $\colim \limits_{n\to \infty} \pi_k (M, m_0^n)$ 
and induce a well-defined function
$$\Psi \co\colim \limits_{n\to \infty} \pi_k (M, m_0^n) \to \rh_k (M).$$
For  $k\geqs 1$, it follows from Lemma~\ref{EH} that $\Psi$ is a homomorphism, while for $k=0$, this is immediate from the definitions  (the monoid structure on 
$\colim \limits_{n\to \infty} \pi_0 (M, m_0^n)$ being that described in Remark~\ref{rmk:pi0}).
Let $\Phi =  \Psi\circ (i \circ f^{-1})$, where $i \circ f^{-1}$ is the map from Lemma~\ref{MS-iso}. We will show that $\wt{\Gamma}$ and $\Phi$ are inverses of one another.\footnote{It then follows that $\Phi$ is independent of $m_0$ and is natural.
Since $i\circ f^{-1}$ is an isomorphism, we also conclude that $\Psi$ is an isomorphism.}

First, consider  
$\wt{\Gamma} \circ \Phi$.
As noted above, each element of $\pi_k (\Omega BM)$ has the form $(f_n)_* \langle \phi \rangle$ for some   $\langle \phi \rangle \in \pi_k (M, m_0^n)$.
Now
\begin{align*}\wt{\Gamma} \circ \Phi \left(  (f_n)_* \langle \phi \rangle \right)  &   =  \wt{\Gamma} \circ \Psi ((i_n)_*  \langle \phi \rangle )    =    \wt{\Gamma} ([\phi])\\
& = J^{-1} \left( [\ol{\gamma(m_0^n)} \boxdot \gamma \circ \phi] \right)  =    (f_n)_* \langle \phi \rangle, \end{align*}
so $\wt{\Gamma}\circ \Phi$ is the identity map.

Next, consider the composition $\Phi \circ \wt{\Gamma}$.  
For each $k\geqs 0$, the group $\rh_k (M)$ is generated by classes of the form $[\phi]$ with $\phi\co S^k \to M$, so it will suffice to check that $\Phi \circ \wt{\Gamma} ([\phi]) = [\phi]$ for each $\phi\co S^k \to M$.
Since $M$ is stably group-like with respect to $[m_0]$, there exists $m\in M$ such that $\phi(1)\bt m$ lies in the path component of $m_0^n$ (for some $n$).  
The maps $\phi$ and  $\phi\bt c_m$ represent the same class in $\rh_k (M)$, so we may assume without loss of generality that $[\phi(1)] = [m_0^n]$, and in fact we may assume $\phi (1) = m_0^n$ since the basepoint $1\in S^k$ is non-degenerate.   
Now 
$$\wt{\Gamma} ([\phi]) = J^{-1} \left(  [\ol{\gamma (\phi (1))} \boxdot (\gamma \circ \phi)] \right) = J^{-1} \left( [\ol{\gamma(m_0^n)} 
\boxdot (\gamma \circ \phi)] \right) =   (f_n)_*  \la \phi \ra.$$
Applying $(i\circ f^{-1})$ to this element gives $(i_n)_*   \la \phi \ra $, which maps to $[\phi]$ under $\Psi$ as desired.
\end{proof}

%%%%%%%%%%%%%%%%%%%%%%%%%%%%%%%
%%%%%%%%%%%%%%%%%%%%%%%%%%%%%%%
%%%%%%%%%%%%%%%%%%%%%%%%%%%%%%%
%%%%%%%%%%%%%%%%%%%%%%%%%%%%%%%

\section{Proof of Theorem~\ref{gen-thm}}\label{sec:proof}

We will need a definition.

\begin{definition}$\label{sub}$
Given a topological monoid $M$ and a submonoid $N \subset \pi_0 (M)$, we define 
\begin{equation}\label{comp}\ol{N} = \{ m\in M\,|\, [m] \in N\}.\end{equation}
More generally, if $S$ is an arbitrary subset of $\pi_0 (M)$, we define 
$$\ol{S} = \ol{\la S \ra},$$ 
where $\la S \ra$ is the submonoid of $\pi_0 (M)$ generated by $S$.
 Note that $\ol{S}$ is a submonoid of $M$, and $\pi_0 \left(\ol{S}\right) = \la S \ra$.  Moreover, $\ol{S}$ is a union of path components of $M$. 
\end{definition}
 
Observe that if $M$ is homotopy commutative,  so is $\ol{S}$ (for every $S\subset \pi_0 (M)$), and if $m\in \ol{S}$ is  strongly 1--anchored in $M$, the same is true in $\ol{S}$.

\begin{proof}[Proof of Theorem~\ref{gen-thm}] 
 
Consider the set $ \mathcal{F}$ of all finite subsets of $\pi_0 (M)$, which forms a directed poset under inclusion.
For each $F\in \mathcal F$, let $\sigma(F) \in \pi_0 (M)$ be the product of all elements in $F$ (this is well-defined, since $M$ is homotopy commutative).
Since the subset of strongly 1--anchored components is cofinal in $\pi_0 (M)$, the set 
$$\mathcal{F}' \defn \{F\in \mathcal{F} \,:\, \sigma (F) \textrm{ is strongly 1--anchored} \}$$
is cofinal in $\mathcal{F}$ (in the sense that each $F\in \mathcal{F}$ is contained in some $F'\in \mathcal{F}'$), 
and hence the natural map
\begin{equation}\label{colim-F}\colim_{F\in \mathcal{F}'} \pi_0 \ol{F} \maps \pi_0 M\end{equation}
is bijective. 

Since $\wt{\Gamma}$ is a natural transformation, to prove the theorem it will suffice to show that $\wt{\Gamma}_{\ol{F}}$ is an isomorphism for each $F\in \mathcal{F}'$ and that the natural maps
\begin{equation}\label{c1}\colim_{F\in \mathcal{F}} \rh_k (\ol{F}) \maps  \rh_k (M)\end{equation}
and
\begin{equation}\label{c2}\colim_{F\in \mathcal{F}} \pi_k (\Omega B\ol{F}) \maps  \pi_k (\Omega BM)\end{equation}
are isomorphisms for each $k$.

To show that $\wt{\Gamma}_{\ol{F}}$ is an isomorphism for each $F\in \mathcal{F}'$, it suffices (by Proposition~\ref{sgl}) to check that $\ol{F}$ is stably group-like with respect to the element $\sigma(F)$. Letting $F = \{a_1, \ldots, a_l\}$,
each component of $\ol{F}$ has the form 
$$C = [a_1^{n_1}\bt a_2^{n_2}\cdots \bt a_{l}^{n_l}]$$
for some  $n_j\geqs 0$ ($j=1, 2, \ldots, l$).
Setting $N = \max \left\{n_1, \ldots, n_l\right\}$,
we have
$$C \bt [a_1^{N-n_1}\bt \cdots \bt  a_l^{N-n_l}] = [\sigma(F)^N],$$
so $\ol{F}$ is stably group-like with respect to $\sigma(F)$.

Next we show that (\ref{c1}) and (\ref{c2}) are isomorphisms.
For the former, it suffices to observe that since (\ref{colim-F}) is a bijection, every map from a path connected space (e.g. $S^k$ or $S^k \cross I$)  into $M$ factors through one of the embeddings $\ol{F}\injects M$.  
For the latter, it suffices to show that for each $k$, these embeddings induce an isomorphism
\begin{equation}\label{colim-eq}\colim_{F\in \mathcal{F}'} \pi_k  (B\ol{F}) \srm{\isom}  \pi_k (BM).\end{equation}

Given a topological monoid $P$, the singular simplicial set  $S. P$ has the structure of  a simplicial object in the category of (discrete) monoids, and we define $B (S. P)$ to be the (thick) geometric realization of the bisimplicial set $N.. (S. P)$ formed by applying the bar construction to each monoid $S_p P$ ($p\in \bbN$); thus the set of simplices of $N.. (S. P)$ in bi-degree $(p,q)$ is $(S_p P)^q \isom S_p (P^q)$.
Let $N. (S. P)$ be the simplicial space
$$q \goesto ||S. ( P^q)||,$$
formed by realizing $N.. (S. P)$ with respect to the index $p$,
so that we have a natural homeomorphism $||N. (S. P)||\isom B (S. P)$ (which we treat as an identification).
The natural maps $||S. (P^q)|| \srm{\heq} P^q$ are weak equivalences~\cite[Lemma 1.11]{ERW}, and
induce a map 
$N. (S.P) \to N. P,$
natural in $P$. This level-wise weak equivalence induces a weak equivalence 
\begin{equation}\label{S.}||B (S. P)|| \srm{\epsilon_P} B P.\end{equation}
by~\cite[Theorem 2.2]{ERW}.

Now consider the commutative diagram
\begin{equation}\label{sing}\xymatrix{ \displaystyle{\colim_{F\in  \mathcal{F}'} \pi_k  B (S. \ol{F})} \ar[r] \ar[d]^-{\colim{(\epsilon_{\ol{F}})_*}}   & \pi_k B (S. M) \ar[d]^{(\epsilon_M)_*}\\
\displaystyle{\colim_{F\in \mathcal{F}'} \pi_k (B \ol{F}) } \ar[r] & \pi_k (BM),}
\end{equation}
in which the vertical maps in (\ref{sing}) are both isomorphisms by (\ref{S.}).

The bottom map in Diagram (\ref{sing}) is the same as (\ref{colim-eq}), and to prove that this map is an isomorphism, it remains to observe that top map in Diagram (\ref{sing}) is an isomorphism.  By (\ref{colim-F}),  each singular simplex $\Delta^n \to M$ factors through one of the embeddings $\ol{F}\to M$, and hence the natural map
$$\colim_{F\in \mathcal{F}'} S. \ol{F} \maps S. M$$
is an isomorphism, as is the induced map
$$\colim_{F\in \mathcal{F}'} B.. (S. \ol{F}) \maps B.. (S. M).$$
This completes the proof, because
homotopy groups of the geometric realization commute with filtered colimits in the category of (bi)simplicial sets  (for the thin geometric realization, this is proven in~\cite[Proposition A.2.5.3]{DGM}, for instance, and the thick and thin geometric realizations are naturally homotopy equivalent by~\cite[Lemma 1.7 and Theorem 7.1]{ERW}). 
\end{proof}

\section{Examples}

%We now apply our main result to several interesting topological monoids.

\subsection{Lawson homology}\label{Lawson-sec}
The above results yield a concrete description of the homology groups for complex projective varieties introduced by Blaine Lawson~\cite{Lawson-homology}.
Let $X$ be a complex projective variety. Then associated to $X$ one has the topological abelian monoid $\mC_m (X)$ of effective $m$--dimensional cycles in $X$. Points in $\mC_m$ are finite formal sums $n_1 V_1 +\cdots +n_k V_k$ with $n_i\in \bbN$ and $V_i \subset X$ an irreducible subvariety of dimension $m$. Formal addition makes $\mC_m$ into an abelian monoid. Fixing an embedding $X\injects \bbC P^n$ allows one to define the degree of a subvariety $V\subset X$, and we can extend by linearity to obtain a homomorphism $\deg\co \mC_m \to \bbN$.
The subsets $\mC_{m,d} = \deg^{-1} (d)$ have the structure of projective varieties, and giving them their analytic topologies makes $\mC_m$ a \e{topological} abelian monoid with the empty sum as identity.  The Lawson homology groups of $X$ are, by definition, the bigraded abelian groups
$$L_m H_n (X) := \pi_{n-2m} \left(\Omega B \mC_m (X)\right).$$
For good introductions to this subject, see the surveys of Lima-Filho~\cite{LF-survey} and Peters--Kosarew~\cite{Peters-Kosarew}.

%Note that   the identity element in $\mC_m (X)$  forms its own path component, so $\mC_m (X)$ is proper and 

Since $\mC_m (X)$ is abelian, Theorem~\ref{gen-thm} gives the following result.

\begin{corollary} \label{L-cor} Let $X$ be a complex projective variety. Then there is a natural isomorphism
$$[S^{n-2m}, \mC_m (X)]/\pi_0 (\mC_m (X)) \srm{\isom} L_m H_n (X).$$
\end{corollary}

\subsection{The deformation representation ring}\label{Rdef-sec}
Next we turn to the deformation representation ring, introduced by Tyler Lawson~\cite{Lawson-simul}.
This is a direct topological analogue of the ordinary representation ring of a discrete group $G$. Consider the topological monoid
$$\ol{\Rep}(G) := \coprod_{n=0}^\infty \Hom(G, U(n))/U(n),$$
where the representation space $\Hom(G, U(n))$ is topologized as a subspace of $\Map(G, U(n))$ (with the compact-open topology), and $U(n)$ acts by conjugation.
Block sum of matrices makes this into a topological abelian monoid, and the unique element in $\Hom(G, U(0))/U(0)$ is the identity element. The deformation representation ring of $G$ is defined by
$$\Rdef_* (G) = \pi_* \left(\Omega B \ol{\Rep} (G)\right).$$
(Lawson shows that there is a ring structure on this graded abelian group, induced by tensor product of representations.) Again, %these monoids are proper, and 
Theorem~\ref{gen-thm} yields:

\begin{corollary} \label{Rdef-cor} For each discrete group $G$, there is a natural isomorphism
$$[S^k, \ol{\Rep} (G)]/\pi_0 \ol{\Rep} (G) \srm{\isom} \Rdef_k (G).$$
\end{corollary}
 
 One can also consider versions of this construction with the unitary groups replaced by other families of Lie groups such as the orthogonal groups or the general linear groups (over $\bbR$ or $\bbC$), with similar results. 
 
\subsection{Deformation $K$--theory}$\label{Kdef-sec}$
We can use Theorem~\ref{gen-thm} to describe the homotopy groups of the unitary and general linear deformation $K$--theory spectra associated to finitely generated discrete groups.  
The homotopy groups of these spectra are simply the ordinary homotopy groups of the group completions of the monoids
$$\Rep(G)_{\hU} = \coprod_{n=0}^\infty \Hom(G, \U(n))_{\hU(n)}$$
and 
$$\Rep(G)_{\hGL} = \coprod_{n=0}^\infty \Hom(G, \GL(n))_{\hGL(n)},$$
where on the right-hand side we are taking the homotopy orbit space (Borel construction) with respect to the conjugation action. Using a functorial model for the universal bundles $EU(n)$ and $EGL(n)$ in forming the Borel construction, block sum of matrices endows the above spaces with the structure of topological monoids. Moreover, every element in each of these monoids is strongly anchored. See~\cite{Ramras-excision} for details.

Theorem~\ref{gen-thm} implies that for each $m\geqs 0$, the natural map
$$\wt{\Gamma} \co \rh_m (\Rep(G)_{\hU}) \maps \K_m (G)$$
is an isomorphism.
Our next goal is to describe these homotopy groups
explicitly 
in terms of spherical families of representations $S^m \to \Hom(G, \U(n))$.

\begin{definition}
Given a discrete group $G$, define
$$\Rep (G) = \Rep(G, \U) \defn \coprod_{n=0}^\infty \Hom(G, \U(n)).$$
Block sum of matrices makes $\Rep(G)$ into a topological monoid, with the unique element in 
$\Hom(G,\U(0))$
 as the  identity.  Replacing $\U(n)$ by $\GL(n)$, we obtain the monoid $\Rep(G, \GL)$.
\end{definition}

\begin{lemma}\label{sa-Hom}
Let $G$ be a discrete group.  Then each component in  $\Rep (G, \U)$ 
is $($strongly$)$ anchored, and the same holds for $\Rep(G, \GL)$.
\end{lemma}
\begin{proof} We will phrase the proof so as to apply to both the unitary and general linear cases.
Consider an $n$--dimensional representation $\rho$.  
Consider the permutation matrices
\begin{equation}\label{Imn}I_{m,n} = \br{\begin{array}{cc} 0_{nm}  & I_n     \\
					 I_m  &  0_{mn}     \\     
	 \end{array}},\end{equation}
where $0_{pq}$ denotes the $p\cross q$ zero matrix.
Note that if $A$ and $B$ are $m\cross m$ and $n\cross n$ matrices, respectively, then $I_{m,n}$ conjugates $A\oplus B$ to $B\oplus A$. Since $I_{m,n}$ is diagonalizable,
Ramras~\cite[Lemma 4.3]{Ramras-excision} implies that there exists a path $A_t$ in $\Stab (\rho\oplus \rho)$ (the stabilizer under the conjugation action) with $A_0 = I_{2n}$ and $A_1 = I_{n,n}$.  Now 
$$(\psi, \psi') \goesto A_t (\psi \oplus \psi') A_t^{-1}$$ 
defines a homotopy anchoring $\rho$.
\end{proof}

Let $K^{-m} (*)$ denote the complex $K$--theory of a point, so $K^{-m} (*) = \bbZ$ for $m$ even and $K^{-m} (*) = 0$ for $m$ odd.

\begin{theorem}$\label{Kdef}$ Let $G$ be a discrete group.  Then there are natural  isomorphisms 
$$ \K_m (G) \isom \rh_m (\Rep (G)) \oplus K^{-m} (*)$$
for each $m > 0$, as well as natural isomorphisms
$$\K_0 (G) \isom  \rh_0 (\Rep(G)) \isom \Gr (\pi_0 \Rep(G)).$$
Analogous  isomorphisms exist in the general linear case, so long as $G$ is finitely generated.
\end{theorem}

We will need a lemma regarding the fibrations
\begin{equation} \label{fibn} \Hom (G, \U(n)) \srm{i_n} \Hom (G, \U(n))_{\hU(n)} \srm{q_n}  B\U(n).\end{equation}

\begin{lemma}\label{bdry}
Let $\rho\co G\to \U(n)$ be a representation of the form $\rho \isom \rho'\oplus I_{k_0}$, and assume that
$\rho'$ decomposes further as a direct sum of the form
$$\rho' = \bigoplus_{i=1}^r k_i \rho'_i$$
for some (non-zero) representations $\rho'_i$, 
where $k_i \rho'_i$ denotes the $k_i$--fold block sum of $\rho'_i$ with itself.
Let $k = \min\{k_0, k_1, \ldots, k_r\}$, 
Then for $m\leqs 2k$, the map $i_n$ appearing in the fibration sequence $(\ref{fibn})$ induces an injection
$$(i_n)_* \co \pi_m (\Hom(G, \U(n)), \rho) \injects \pi_m ( \Hom (G, \U(n))_{\hU(n)}, i_n (\rho)).$$

The analogous statement holds with $\U(n)$ replaced by $\GL (n)$, so long as $G$ is finitely generated.   
\end{lemma}
\begin{proof} We begin by proving the unitary case. At the end, we will explain the additional arguments needed in the general linear case.

If $k=n$ (that is, if $\rho = I_n$), then this follows from the fact that the fibration (\ref{fibn}) admits a splitting sending $[e, *]\in B\U(n)\isom\{*\}_{\hU(n)}$ to $[e, I_n] \in \Hom (G, \U(n))_{\hU(n)}$; this splitting is just the map on homotopy orbit spaces induced by the $\U(n)$--equivariant inclusion $\{I_n\} \injects \Hom(G, \U(n))$ (note that in this case, $(i_n)_*$ is injective in all dimensions, not just when $m\leqs 2k=2n$).
So we will assume $k<n$.

It will suffice to show that  for $m\leqs 2k$, the boundary map  
$$\partial \co \pi_{m+1} (B\U(n), *) \maps \pi_m (\Hom(G, \U(n)), \rho)$$
is zero.  Since we have assumed $k < n$, we have  $m \leqs 2k < 2n$, which means $\pi_{m+1} (B\U(n), *)$ is in the stable range, and is zero for $m$ even.  Hence we will assume $m$ is odd for the remainder of the proof.

Let $O_\rho \subset \Hom(G, \U(n))$ denote the conjugation orbit of $\rho$.  Then $\partial$ factors through the boundary map for the fibration
$$O_\rho \maps (O_\rho)_{\hU(n)}\maps B\U(n),$$
and we claim that $\pi_m (O_\rho) = 0$ if $m$ is odd and  less than $2k$, which will complete the proof.  Letting $\Stab_\rho \leqs \U(n)$ denote the stabilizer of $\rho$ under conjugation, Schur's Lemma implies that 
\begin{equation}\label{stab}\Stab_\rho \isom \prod_i \U(n_i),
\end{equation}
where the numbers $n_i$ are the multiplicities of the irreducible summands of $\rho$.  Our assumption on $\rho$ implies that $n_i\geqs k$ for each $i$.

The inclusion 
$$\Stab(I_{k_0}) \isom \U(k_0)  \injects \Stab_\rho  \injects \U(n)$$ is homotopic to the standard inclusion of $\U(k_0)$ into $\U(n)$, and hence is an isomorphism on homotopy in dimensions less than  $2k_0$, and in particular, this map is an isomorphism on homotopy in dimensions less than  $2k$.  
This implies that the inclusion
$\Stab_\rho \injects \U(n)$
is surjective on homotopy in dimensions less than $2k$, and consequently the boundary map
for the fibration sequence
$$\Stab_\rho \maps \U(n) \srm{q} O_\rho$$
is an injection
$$\pi_m (O_\rho) \injects \pi_{m-1} (\Stab_\rho )$$
 for $m\leqs 2k-1$.  But we are assuming that $m$ is  odd and $m<2k\leqs 2n_i$ 
 for each $i$, so
$$\pi_{m-1} \Stab_\rho \isom \prod_i \pi_{m-1} \U(n_i) = 0.$$
It follows that $\pi_m (O_\rho) = 0$ when $m$ is odd and less than $2k$.  This completes the proof in the unitary case.

To extend this argument to the general linear case, we need to analyze $\Stab_\rho$ for representations $\rho\co G\to \GL(n)$, with $G$ finitely generated.
If $\rho$ is completely reducible (that is, isomorphic to a block sum of irreducible representations) 
then a decomposition analogous to (\ref{stab}) still holds, but this can fail if $\rho$ is not completely reducible.
We claim there exists a representation $\rho^{\textrm{red}}$ in the path component of $\rho$ which is completely reducible and still satisfies the hypotheses of the Lemma.  This will suffice, since the map $(i_n)_*$ is independent (up to isomorphism) of the chosen basepoint within the path component of $\rho$.  Note that the rest of the unitary argument applies equally well in the general linear case, since $\GL(n) \heq \U(n)$.

To prove the existence of $\rho^{\textrm{red}}$, we appeal to some results from algebraic geometry.
If $G$ is generated by, say, $l$ elements ($l < \infty$), then  
$$\Hom(G, \GL (n)) \subset \GL (n)^l \isom \{(A,B) \in M_{n\cross n} \bbC \isom \bbC^{n^2} : AB = I\}^l$$
is an affine variety, cut out   by the ideal of polynomials corresponding to the relations in $G$.
A basic result in Geometric Invariant Theory states for every conjugation orbit   $O \subset \Hom (G, \GL (n))$, there exists a (unique) completely reducible representation inside the closure $\ol{O}$ (in general, orbit closures of complex reductive groups acting on affine
 varieties contain unique closed orbits~\cite[Corollary 6.1 and Theorem 6.1]{Dolgachev}, and in the present situation complete reducibility is equivalent to having a closed orbit~\cite[Theorem 1.27]{LM}).  This implies that every path component of $\Hom (G, \GL(n))$ contains a completely reducible representation, since $ \Hom (G,\GL (n))$ is triangulable (as are all affine varieties~\cite{Hironaka}). 
 
 Recall that we have a decomposition 
$$\rho =\left( \bigoplus_i k \rho'_i \right) \oplus I_k.$$
Letting $\rho''_i$ denote the completely reducible representation in the path component of 
$\rho'_i$, we see that 
$$\rho^{\textrm{red}} :=\left( \bigoplus_i k \rho''_i \right) \oplus I_k$$
 is again completely reducible,   lies in the same path component as $\rho$, and satisfies the hypotheses of the Lemma, as desired.
\end{proof}

\begin{proof}[Proof of Theorem~\ref{Kdef}]  For $m=0$, this is elementary (see the proof of~\cite[Lemma 2.5]{Ramras-excision}), so we will assume $m>0$.  We will  work in the unitary case; the proof in  the general linear case is the same.

Lemma~\ref{sa-Hom}, Theorem~\ref{gen-thm}, and Proposition~\ref{univ} give natural isomorphisms
$$\pi_m \left(\Omega B \Rep (G) \right) \isom \rh_m (\Rep(G)) \isom [S^m, \Rep(G)]/\pi_0 \Rep (G).$$
To obtain natural isomorphisms between
$$ \K_m (G) = \pi_m \left( \Omega B \left(\Rep(G)_{\hU}\right)\right)$$
and
$$\rh_m(\Rep(G)) \oplus K^{-m} (*),$$
($m > 0$) it remains only to show that there are natural isomorphisms
\begin{equation}\label{split}\pi_m \left(\Omega B \left(\Rep(G)_{\hU}\right) \right)\isom 
\pi_m \left( \Omega B \left(\Rep(G)\right)\right)  \oplus K^{-m} (*).
\end{equation}

Each representation space for the trivial group is a single point, so
$$\Rep(\{1\})_{\hU} \isom \coprod_n B\U(n),$$
and this monoid is stably group-like with respect to each of its components.
Hence the Group Completion Theorem, together with Bott Periodicity, gives us isomorphisms
$$\pi_m \left(\Omega B\left( \Rep(\{1\})_{\hU}  \right) \right)\isom K^{-m} (*),$$
and this will be our model for $K^{-m} (*)$.

The  inclusion $\{1\} \injects G$ induces a map of monoids
\begin{equation}\label{q} q \co \Rep(G)_{\hU} \maps \Rep(\{1\})_{\hU},\end{equation}
and hence a map 
$$\Omega B q\co \Omega B \left( \Rep(G)_{\hU}\right) \maps \Omega B\left( \Rep(\{1\})_{\hU}  \right).$$
The maps
\begin{equation}\label{in} i_n \co  \Hom(G, \U(n))\injects \Hom(G, \U(n))_{\hU(n)}
\end{equation}
combine into a monoid homomorphism
$$i\co \Rep(G)\injects \Rep(G)_{\hU},$$
and we have an induced map 
$$\Omega B i\co \Omega B \Rep (G) \maps\Omega B\left( \Rep(G)_{\hU}  \right).$$  
To complete the proof, it will suffice to show that the sequence 
\begin{equation}\label{ses} 0\xmaps{} \pi_m \left(\Omega B \Rep (G) \right)\xmaps{i_*}  \pi_m \left( \Omega B\Rep(G)_{\hU}\right) \xmaps{q_*}
 \pi_m \left( \Omega B \Rep(\{1\})_{\hU}\right)\xmaps{} 0
 \end{equation}
 is split exact for each $m\geqs 1$ (note that we are abbreviating $(\Omega B i)_*$ to $i_*$ and $(\Omega B q)_*$ to $q_*$).
The map $q_*$ admits a right inverse, induced by the projection $G\to \{1\}$, so the sequence (\ref{ses}) splits and $q_*$ is surjective. 
Next, the composite $q\circ i$ factors through the discrete monoid $\bbN$.  Since $\pi_m ( \Omega B\bbN) = 0$ for $m > 0$, we have $q_* \circ i_* = 0$.  To prove exactness of (\ref{ses}), it remains to show that $\ker(q_*) \subset \Img (i_*)$ and $\ker(i_*) = 0$.  We will prove these statements by directs argument using Theorem~\ref{gen-thm}, which provides an isomorphism between (\ref{ses}) and the corresponding sequence obtained by applying the functor $\rh_m$ in place of $\pi_m (\Omega B)$.  Hence for the rest of the argument  we will work 
with the sequence
\begin{equation}\label{ses'} 0\maps \rh_m   \Rep (G) \srm{i_*}  \rh_m\left(\Rep(G)_{\hU}\right) \srm{q_*}
 \rh_m\left( \Rep(\{1\})_{\hU}\right) \maps 0.
 \end{equation}

First we show that $\ker(q_*) \subset \Img (i_*)$.    
By Proposition~\ref{univ}, 
each element in $\ker (q_*)$ is represented by a map $\psi \co S^m \to \Hom(G, \U(k))_{\hU(k)}$ such that for some constant map $c\co S^m\to B\U(k')$, the map
$(q\circ \psi) \oplus c$ is nullhomotopic.  It follows that for any constant map $\wt{c} \co S^m \to \Hom(G, \U(k'))_{\hU(k')}$, the homotopy class  
 $\la \psi \oplus \wt{c}\ra$ lies in the kernel of   
$$q_*\co \pi_* \left( \Hom (G, \U(k+k'))_{\hU(k+k')} \right) \maps \pi_* B\U(k+k')$$
 (for appropriately chosen basepoints).  
Since the sequence (\ref{fibn})
is a fibration sequence for each $n\geqs 0$, there exists a map 
$$\rho\co S^m \to \Hom(G, \U(k+k'))$$
 such that $i_n\circ \rho\heq \psi \oplus \wt{c}$, 
and this shows that $[\psi]$ is in the image of the map $i_*$ in (\ref{ses'}).

The proof that $\ker(i_*) = 0$ is similar, but will require Lemma~\ref{bdry}.  Each element in $\ker(i_*)$ is represented by a map $\rho\co S^m\to \Hom(G, \U(n))$ such that 
\begin{equation}\label{psi} (i\circ \rho) \oplus d \heq  c  \end{equation} 
for some  constant maps $c, d \co S^m \to \Rep(G)_{\hU}$. 
Let
$m\rho(1) \oplus I_m$ denote the constant map $S^m\to \Rep(G)$ with image $m\rho(1) \oplus I_m$.
Adding 
$$md \oplus mc \oplus i ((m-1)\rho(1) \oplus I_m),$$
to both sides of (\ref{psi}) gives
\begin{eqnarray}\label{big}(i\circ \rho) \oplus (m+1)d \oplus    mc \oplus i ((m-1)\rho(1) \oplus I_m)\\
\heq  (m+1)c \oplus  md   \oplus i ((m-1)\rho(1) \oplus I_m).\notag
\end{eqnarray}

Since $\U(n)$ is path connected, the map $i_n$ induces a bijection on path components for every $n$.  Thus there exist constant maps $\wt{c}$  and $\wt{d}$ such that  
\begin{equation}\label{pr}i\circ \wt{c} \heq c\,\,\, \textrm{ and } \,\,\, i\circ \wt{d} \heq d.
\end{equation}
Setting 
$$e = (m+1)\wt{d} \oplus m\wt{c} \oplus (m-1)\rho(1) \oplus I_m,$$
Equation (\ref{big}) gives
\begin{eqnarray}i \circ (\rho \oplus e) 
\heq i \circ ((m+1)\wt{c} \oplus m\wt{d}  \oplus (m-1)\rho(1) \oplus I_m), 
\end{eqnarray}
so in particular $i \circ (\rho \oplus e)$
is nullhomotopic. (Note that while the above nullhomotopy need not be based, 
a based map out of a sphere that is nullhomotopic in the unbased sense is also nullhomotopic in the based sense, since basepoint in the sphere is non-degenerate.)
Moreover, we have an isomorphism of representations
$$\rho(1) \oplus e(1) \isom m\rho(1)\oplus (m+1) \wt{d}(1) \oplus m \wt{c}(1) \oplus I_m,$$
so, letting $N=\dim(\rho \oplus e)$, Lemma~\ref{bdry} implies that
 the map
$$\pi_m  \left(\Hom(G, \U(N))\right)\xmaps{(i_N)_*} \pi_m   \left(\Hom(G, \U(n))_{\hU(N)}\right)
$$
 is injective if we use $(\rho \oplus e) (1)$ as our basepoint for $ \Hom(G, \U(N))$.
Since $i \circ (\rho \oplus e)$
is nullhomotopic,  we conclude that $\rho \oplus e$ is nullhomotopic as well, and since $e$ is constant it follows that $[\rho] = 0$ in $\rh_m \Rep(G)$.  This completes the proof that (\ref{ses'}) is exact, and also completes the proof of the Theorem.
\end{proof}

\begin{question} Does the analogue of Theorem~\ref{Kdef} hold for representations valued in $O(n)$ or $\GL(n, \bbR)$?
\end{question}

\section{Applications to families of characters}$\label{char-sec}$
In this section, all discrete groups will be assumed to be finitely generated.
We now give some applications of the preceding results to spherical families of (unitary) characters, that is, maps
$$\chi\co S^k \maps \Hom(G, U(n))/U(n).$$
We introduce two basic properties of such families.

\begin{definition} Consider a family of characters 
$$\chi \co X \maps \Hom(G, U(n))/U(n).$$

We say that $\chi$ is \e{stably nullhomotopic} if there exists a character $\chi_0 \in \Hom(G, U(n))/U(n)$ such that $\chi \oplus \chi_0$ is nullhomotopic.

We say that $\chi$
is \e{geometric} if there is a family of representations
$$\rho \co X \to \Hom(G, U(n))$$
such that $\chi \heq \pi \circ \rho$, where $\pi \co \Hom(G, U(n))\to \Hom(G, U(n))/U(n)$ is the quotient map. We say that $\chi$ is \e{stably geometric} if there exists a character $\chi_0$ such that $\chi\oplus \chi_0$ is geometric.
\end{definition}

The term \e{geometric} is motivated by the observation that when $G$ is the fundamental group of a manifold $M$, a family $\chi$ is geometric if and only if it is homotopic to the holonomy of some family of flat bundles over $M$ (parametrized by $X$).

Our descriptions of $\pi_* \K(G)$ and $\pi_* \Rdef (G)$ lead to close relationships between these concepts and the \e{Bott map} on deformation $K$--theory. To work with this map, we will view $\K(G)$ as an $\bS$--module in the sense of~\cite{EKMM}; see~\cite[Section 6]{Ramras-stable-moduli} for a discussion of the necessary concepts regarding $\bS$--modules and their applications in this context. The main point we need is that there are $\bS$--modules $\Rdef (G)$ and $\K(G)$  whose homotopy groups are naturally isomorphic to  $\Rdef_* (G)$ and $\K_* (G)$ (respectively) as defined in Sections~\ref{Rdef-sec} and~\ref{Kdef-sec}.

The Bott map takes the form 
$$\beta \co \Sigma^2 \K(G)\to \K(G),$$
and is induced from the classical Bott map $\Sigma^2 \ku \to \ku$ on the connective $K$--theory spectrum $\ku = \K(1)$ 
by smashing (over $\ku$) with the identity map on $\K (G)$. Note here that tensor product of representations makes $\K (G)$ into a \e{ring spectrum} (this was first proven in~\cite{Lawson-prod}) and in fact an $\bS$--algebra. The 
 natural map $q^* \co \K(1) \to \K(G)$ (where $q\co G\to 1$ is the projection) then gives $\K(G)$ the structure of a $\ku$--module (in fact, a $\ku$--algebra).

\begin{lemma}\label{Bott} Let $G$ be finitely generated.
Then the Bott map 
\begin{equation} \label{beta} \beta_* \co  \K_{m-3} (G) \to \K_{m-1} (G)\end{equation}
is injective
if and only if
every $S^m$--family of characters of $G$ is stably geometric. Furthermore, every  $S^m$--family of characters of $G$ is stably nullhomotopic if and only if $\Rdef_m (G) = 0$, which holds if and only if (\ref{beta}) is injective \e{and} $\beta_* \co \K_{m-2} (G) \to \K_{m} (G)$ is surjective.
\end{lemma}
\begin{proof} The statements are trivially true for $m=0$, so we assume $m>0$.

Corollary~\ref{Rdef-cor} and Proposition~\ref{univ} show that vanishing of $\Rdef_m (G)$ is equivalent to the statement that every family $S^m \to \Hom(G, U(n))/U(n)$ is stably nullhomotopic. To relate vanishing of $\Rdef_m (G)$ to the Bott map, we use T. Lawson's homotopy cofiber sequence~\cite{Lawson-simul}
$$\Sigma^2 \K(G)\srm{\beta} \K(G) \maps \Rdef (G),$$
where the second map is induced by the quotient map  
$\Rep(G)_{hU} \to \ol{\Rep} (G)$. The associated long exact sequence in homotopy has the form
\begin{align*} \cdots \to \pi_{m} \Sigma^2 \K(G) = \K_{m-2} (G) \xmaps{\beta_*} \K_m (G) \to \Rdef_m (G) \hspace{.5in}\\
\hspace{.5in} \to \pi_{m-1} \Sigma^2 \K(G) = \K_{m-3} (G) \xmaps{\beta_*} \K_{m-1} (G)\to \cdots.
\end{align*}
This yields the desired statement regarding vanishing of $\Rdef_m (G)$, and shows that 
shows that injectivity of $\beta_* \co  \K_{m-3} (G) \to  \K_ {m-1}(G)$ is equivalent to surjectivity 
of $\K_m (G) \to   \Rdef_m (G)$. It remains to show that surjectivity of this map is equivalent to the first statement in the lemma. 

Theorem~\ref{Kdef} provides a splitting  
$$\K_m (G) \isom \rh_m (\Rep (G)) \oplus \pi_m \ku,$$
 in which the second summand 
is the image of the natural map $\ku \to \K (G)$ (on homotopy). Since the Bott map on $\pi_* \ku$ is an isomorphism, this summand is in the image of the Bott map on $ \K_*(G)$, and hence maps to zero under the projection $\K_* (G)\to \Rdef_* (G)$. 
Together with the naturality statement in Theorem~\ref{gen-thm}, this implies
that surjectivity of 
$$\K_m (G) \to  \Rdef_m (G)$$ 
is equivalent to  surjectivity of
$$\rh_m (\Rep (G)) \maps \rh_m (\ol{\Rep} (G)),$$
and the desired result follows from Proposition~\ref{univ}.
\end{proof}

\begin{proposition}\label{d<2} If $d\leqs 2$ and $G$ is finitely generated, then every spherical family of characters $S^d \to \Hom(G, U(n))/U(n)$ is stably geometric.
\end{proposition}
\begin{proof} This follows immediately from Lemma~\ref{Bott}, because the domain of the Bott map $\K_{d-3} G\to \K_{d-1} G$ is zero.
\end{proof}

In the examples known thus far, two interesting phenomena appear. First, \e{all} spherical families of characters are stably geometric. Second, once the dimension of the spherical family is  sufficiently high, all spherical families of characters are stably nullhomotopic. The dimension bound, interestingly, appears to be  related to the \e{rational cohomological dimension} of $G$, namely the smallest dimension in which 
$H^* (BG; \bbQ)$ is non-zero. We denote this number by $\qcd(G)$.

We now explain several examples of these phenomena. It will also be helpful to understand the structure of $\K(G)$ as a $\ku$--module in these examples.

\

\nd{\bf Finite groups.} If $G$ is finite, then $\Hom(G, U(n))/U(n)$ is a discrete space (it is a subspace of $U(n)^G$, hence compact and Hausdorff, and it is finite since $G$ has finitely many irreducible representations). Hence every family of characters parametrized by a connected space is constant, and in particular both stably nullhomotopic and stably geometric. Note that in this case, $H^* (BG; \bbQ) = 0$. 

As a $\ku$--module, $\K (G)$ is simply a wedge of copies of $\ku$, one for each irreducible representation of $G$. This was first proven by T. Lawson~\cite[Chapter 6.3]{Lawson-thesis}; see also~\cite[Section 5]{Ramras-excision}.
 
\

\nd{\bf  The Heisenberg group.} Let $H$ denote the (integral) Heisenberg group, consisting of $3\cross 3$ upper-triangular integer matrices with 1's on the diagonal. T. Lawson calculated both $\Rdef_* (H)$ and $\K_* (H)$~\cite{Lawson-thesis, Lawson-prod}. In particular, $\Rdef_* (H)$ vanishes for $*>2$.   Proposition~\ref{d<2} then implies that   all spherical families of characters of $H$ are stably geometric, and Corollary~\ref{Rdef-cor} implies that 
those of dimension greater than 2 are stably nullhomotopic. Note that $H$ has rational cohomological dimension 3, because $\bbR^3/H$ is a closed, orientable 3--manifold that models $BH$ (here we view 
$\bbR^3$ as the Lie group of $3\cross 3$ upper-triangular \e{real} matrices with 1's on the diagonal).

We in fact have the following description of $\K(H)$ as a $\ku$--module:
$$\K(H) \heq   \left(\bigvee_{1}^\infty \ku \right) \vee \left(\bigvee_{1}^\infty \Sigma \ku \right)$$ 
This weak equivalence is obtained by choosing bases $\{f_i\}_i$ and $\{g_i\}_i$ for   $\K_0 (H)\isom \bigoplus_{1}^\infty \bbZ$ and $\K_1 (H)\isom \bigoplus_{1}^\infty \bbZ$  (respectively). These classes can be viewed as maps of spectra
$f_i \co \bS^0 \to \K(H)$ and $g_i \co \bS^1\to \K(H)$. We now obtain maps $\wt{f_i} \co \ku \to \K(H)$ as the composites
$$\bS^0 \sm  \ku \xmaps{f_i \sm \textrm{id}} \K(H) \sm \ku \srm{\mu}  \K(H),$$
where $\mu$ is the structure map for the $\ku$--module $\K(H)$, and similarly 
we obtain maps $\wt{g_i} \co \Susp \ku \to \K(H)$ as the composites
$$\bS^1 \sm  \ku \xmaps{g_i \sm \textrm{id}} \K(H) \sm \ku \srm{\mu}  \K(H).$$
The map $(\bigvee_i f_i) \vee (\bigvee_i g_i)$ is now a map of $\ku$--modules and an isomorphism on homotopy, as desired. (Note that these infinite wedges should be interpreted as colimits of finite wedges.)

\

\nd {\bf  Crystallographic groups.}
Recall that a ($d$--dimensional) crystallographic group is a discrete, cocompact subgroup of the group of isometries of $\bbR^d$ (with its standard Euclidean metric).
In~\cite{Ramras-crystal}, it was shown that $\Rdef_* (G)$ vanishes in degrees $*>d$. 
Combining this with Corollary~\ref{Rdef-cor} leads to the following result.

\begin{proposition} \label{crystal} If $G$ is a $d$--dimensional crystallographic group, then every $S^k$--family in $\Hom(G, U(n))/U(n)$, with $k>d$, is stably nullhomotopic.
\end{proposition}

For the 17 isomorphism classes of 2-dimensional crystallographic groups, it follows that the Bott map is injective in all dimensions. 

\begin{proposition} If $G$ is a 2-dimensional crystallographic group, then all spherical families of characters of $G$ are stably geometric.
\end{proposition}

T. Lawson calculated the deformation $K$--theory of the 2-dimensional crystallographic group 
$G = \bbZ^2 \rtimes \bbZ/4\bbZ$, where $\bbZ/4\bbZ$ acts by an order 4 signed permutation matrix, concluding that $\K (G) \heq (\bigvee_8 \ku) \vee \Susp^2 \ku$.

For many crystallographic groups, $\Rep(G)$ is stably group-like with respect to the trivial representation (see~\cite[Section 10]{Ramras-crystal}), and hence one only needs to stabilize with respect to block sum by identity matrices.

\

\nd{\bf  Surface groups.} Let $\Sigma$ be a \e{product} of compact aspherical surfaces, possibly with boundary. The deformation $K$--theory of $\pi_1 \Sigma$ was calculated in~\cite{Ramras-stable-moduli}. In particular, the Bott map on $\K_* (\pi_1 \Sigma)$ is injective in \e{all} dimensions, and $\Rdef_* (\pi_1 \Sigma)$ vanishes above the \e{rational cohomological dimension} of $\Sigma$~\cite[Proposition 6.11]{Ramras-stable-moduli}. Thus all spherical families of characters are stably geometric, and those of dimension greater than $\qcd (\Sigma)$ are stably nullhomotopic. Note that $\Rep (\pi_1 \Sigma)$ is stably group-like with respect to the trivial representation~\cite[Lemma 6.4]{Ramras-stable-moduli}, so  one only needs to stabilize with respect to block sum by identity matrices.

\begin{proposition}\label{prop:surf} With $\Sigma$ as above, for each spherical family of characters of $\rho \co S^k \to \Hom(\pi_1 \Sigma, U(n))/U(n)$, there exists $N\in \N$ such that $\rho\oplus I_N$ is geometric.
\end{proposition}

We note that this class of groups includes free groups (since $\Sigma$ is allowed to have boundary) and free abelian groups. In these cases
 there are deformation retractions
$$\Hom(\pi_1 \Sigma, \GL(n))/ \GL(n) \srm{\heq} \Hom(\pi_1 \Sigma, U(n))/U(n).$$
Indeed, it is shown in Florentino--Lawton--Ramras~\cite[Proposition 3.4]{FLR} that  $\Hom(\pi_1 \Sigma, \GL(n))/ \GL(n)$ (which is usually non-Hausdorff) deformation retracts to the geometric invariant theory quotient $\Hom(\pi_1 \Sigma, \GL(n))/\!\!/ \GL(n)$, which in turn deformation retracts to the subspace of unitary characters by Florentino--Lawton~\cite{FL-free, FL-abelian} (see also Bergeron~\cite{Bergeron}).
Hence in these cases, Proposition~\ref{prop:surf} applies to spherical families of general linear representations as well.
 
 \

These examples raise the following questions.

\begin{question} Does there exist a finitely generated discrete group $G$ and a spherical family of unitary characters that is \e{not} stably geometric?
\end{question}

\begin{question}\label{null-q} If $G$ is a finitely generated discrete group, and $d>\qcd (G)$, are all $S^d$--families of  unitary characters are stably nullhomotopic?
\end{question}

It would also be interesting to answer these questions for other classical  sequences of Lie groups (e.g. the general linear groups) in place of $U(n)$.

\s{.25}

We close by noting that these phenomena interact well with both free and direct products of groups. 

For free products, the main result of~\cite{Ramras-excision} provides short exact sequences
$$0 \maps \pi_m \ku \maps \K_m (G_1) \oplus \K_m (G_2) \maps \K_m (G_1*G_2)\maps 0$$
for each $m\geqs 0$, and the maps in this sequence commute with the Bott maps. By the 4-lemma, injectivity of the Bott maps for $G_1$ and $G_2$ implies injectivity of the Bott map for $G_1*G_2$. 
In all cases in which injectivity of the Bott map is known  (finite groups, products of surface groups, the Heisenberg group, $G = \bbZ^2 \rtimes \bbZ/4\bbZ$) we in fact have a more detailed description of $\K(G)$ as a $\ku$--module: it is \e{elementary}, in the following sense.

\begin{definition} We say that a $\ku$--module $M$ is \e{elementary} if it is a countable wedge sum of $\ku$--modules of the form $\Susp^i \ku$ and $\Susp^i (\ku/n)$. The maximum value of $i$ appearing in this decomposition will be called the \e{suspension degree} of $M$.
\end{definition}

The spectrum $\ku/n$ is the ``cofiber of multiplication by $n$''; its homotopy is given by $\pi_* (\ku/n) = \bbZ/n\bbZ$ for $*$ even, and $\pi_* (\ku/n) = 0$ for $*$ odd, and its Bott map is an isomorphism in all (positive) dimensions. The spectrum $\ku/2$ appears as a summand in the deformation $K$--theory of non-orientable surface groups; see~\cite{Ramras-stable-moduli}. 

We have the following basic observations about elementary modules.

\begin{lemma}\label{elem-lemma} Assume $\K(G)$ is elementary. Then its Bott map is injective in all dimensions, and $\Rdef_m (G)$ is isomorphic to the cokernel of 
$$\beta_* \co \K_{m-2} (G) \to \K_m (G).$$ 
Consequently, all spherical families of characters of $G$ are stably geometric, and the maximum dimension in which $\Rdef_* (G)$ is non-zero is equal to the suspension degree of $\K(G)$.
\end{lemma}

\begin{proposition}\label{geom} If $G$ is built via free and direct products from groups whose deformation $K$--theory is elementary, then $\K(G)$ is elementary as well, and hence all spherical families of characters of $G$ are stably geometric. 

Furthermore, the suspension degree of  $\K(G_1 * G_2)$ is the maximum of the suspension degrees of $\K(G_1)$  and  $\K(G_2)$.
\end{proposition}
\begin{proof}
The main result of~\cite{Lawson-prod} provides a weak-equivalence of $\ku$--modules
\begin{equation}\label{prod-form}\K (G_1 \cross G_2) \heq \K(G_1) \sm_{\ku} \K(G_2),\end{equation}
where the smash product on the right is calculated in the derived category of $\ku$--modules.  
Since the smash product is a left adjoint~\cite[III.3.2]{EKMM}, it commutes with colimits, and in particular it distributes over (countable) wedge sums.
By~\cite[Lemma 6.7]{Ramras-stable-moduli}, there is a weak equivalence of $\ku$--modules
\begin{equation}\label{kun}(\ku/n)\sm_\ku (\ku/m) \heq \ku/\textrm{gcd}(n,m) \vee \Susp (\ku/\textrm{gcd}(n,m)).\end{equation}
It follows that if $\K(G_1)$ and $\K(G_2)$ are elementary, so is $\K(G_1 \cross G_2)$.

For free products, the main result of~\cite{Ramras-excision} shows that 
$\K(G_1*G_2)$ is the homotopy pushout of the diagram
\begin{equation}\label{free}\K(G_1) \longleftarrow \ku \maps \K(G_2),\end{equation}
where the maps are induced by the projections $G_i\to 1$. These maps $\ku\to \K(G_i)$ are split by the map induced by $1 \to G$, and hence they are isomorphisms onto $\ku$--summands in $\K(G_i)$. Writing $\K(G_1) \heq \ku \vee M_1$ and $\K(G_2) \heq \ku \vee M_2$, it follows that the homotopy pushout of (\ref{free}) is weakly equivalent to $\ku \vee M_1 \vee M_2$, which is again elementary and has the claimed suspension degree.
\end{proof}

\begin{corollary} If  Question~\ref{null-q} has a positive answer for $G_1$ and for $G_2$, and $\K(G_1)$ and $\K(G_2)$ are elementary, then the question also has a positive answer for $G_1 *G_2$.
\end{corollary}
\begin{proof} By Lemma~\ref{elem-lemma}, we know that the suspension degree of $\K(G_i)$ is at most $\qcd (G_i)$ ($i=1,2$), and we must show that the same holds for $G_1*G_2$.
By Proposition~\ref{geom}, the suspension degree of $\K(G_1 * G_2)$ is the maximum of those for $\K (G_1)$ and $\K (G_2)$.
Similarly,   the Mayer--Vietoris sequence applied to $B(G_1 * G_2) =  BG_1  \vee BG_2$ shows that $\qcd (G_1 * G_2) = \max (\qcd (G_1), \qcd (G_2))$. 
\end{proof}

For direct products, the situation is somewhat subtler, because the additional suspension appearing in (\ref{kun}) can, in principle, prevent the suspension degree from being additive under direct products. However, the suspension degree \e{is} additive for elementary modules of the form 
$$\left( \bigvee_j \Susp^{i_j}  \ku \right) \vee \left(\bigvee_k \Susp^{l_k} (\ku/n_k)  \right)$$ 
that further satisfy $\max\{i_j\}_j > \max\{l_k\}_k$, and 
in all the cases considered above in which $\K(G)$ is known to be elementary, this holds: terms of the form $\Sigma^l (\ku/n)$ only appears for  fundamental groups of aspherical, non-orientable surfaces, whose deformation $K$--theory has the form 
$$\ku \vee \left(\bigvee_i \Susp \ku \right) \vee \ku/2$$ 
with $i>1$. Thus if $G$ is a direct product of the groups above for which we know that $\K(G)$ is elementary, then we have a positive answer to Question~\ref{null-q} for $G$.

 \appendix

\section{Abelian monoids}\label{LH-sec}

In this Appendix we consider the case of abelian monoids $A$, in which every element is automatically strongly anchored. 
In this case we can give a different proof of Theorem~\ref{gen-thm} (under a mild point-set assumption), using Quillen's approach to group completion~\cite[Appendix Q]{Friedlander-Mazur}. The argument presented here is similar to the arguments of Friedlander--Mazur~\cite[Section 2]{Friedlander-Mazur} (see also Friedlander--Lawson~\cite[Section 1]{FL}). 
%At the end of this section, we discuss two particularly interesting abelian monoids, giving rise to \e{Lawson homology} of  complex projective varieties, and to the \e{deformation representation ring} of a discrete group.

In this section, $BA$ will denote the \e{thin} realization $|N.\underline{A}|$ of the bar construction on $A$, and we will 
assume that the inclusion of the identity element into $A$ is a closed cofibration. As discussed in the Introduction, 
this condition, which we call \e{properness}, implies that the natural map $||N.\underline{A}||\to |N.\underline{A}|$ is a homotopy equivalence.

Let $S$ be a set, and let $F(S)$ be the free abelian monoid generated by $S$. 
We may consider $F(S)$ as a category whose object set is the underlying set of $F(S)$, and whose morphisms $x\to y$ are those $z\in F(S)$ satisfying $x+z=y$. This is a filtered category, since given two objects $x,y\in F$, there exist morphisms $x\to x+y$ and $y\to y+x = x+y$ (and any two parallel arrows are equal).
In what follows, we will use the fact that
homotopy groups commute  with filtered colimits of simplicial sets~\cite[Proposition A.2.5.3]{DGM}.

Let $A$ be a topological abelian monoid. Following \cite{Friedlander-Mazur}, we define a \e{base system} for $A$ to be a set $\mathcal{B}$ of representatives for $\pi_0 A$ (in fact, Friedlander and Mazur allow more general base systems).
Given an arrow $a\to a'$ in $F(\mathcal{B})$, there exists a unique $b\in F(\mathcal{B})$ such that $a' = a+b$, and we
say that this arrow is \e{labeled by} $b$.

A choice of base system $\mathcal{B}$ defines a functor
 from $F(\mathcal{B})$ to simplicial sets,  sending each object to $S. A$ (the singular set of $A$), and
 each morphism labeled by $b$ to the map  $S. A \to S. A$ induced by multiplication by $b\in A$.
 We will denote the colimit of this functor simply by $\colim_\mathcal{B} S. A$, and similarly for the functor obtained by composing with geometric realization.

\begin{theorem}\label{Ab-thm} Let $A$ be a proper  topological abelian monoid, and let $\mathcal{B}\subset A$ be a base system for $A$. 
Then there are  weak equivalences
$$\colim_\mathcal{B} |S. A| \isom |\colim_\mathcal{B} S. A |\srm{\heq} |\Gr (S. A)| \heq \Omega BA,$$
where $\Gr (S. A)$ denotes the level-wise group completion of the simplicial $($discrete, abelian$)$ monoid $S. A$. 
\end{theorem}

\begin{proof} We begin by noting that since $A$ is abelian, so is $S_k A$ (for each $k$), and hence $S_k A$ and $S. A$ are \e{good} in the sense of~\cite[Appendix Q]{Friedlander-Mazur}.

We first construct the weak equivalence
$$ |\Gr (S. A)| \heq \Omega BA,$$
which will be induced by a natural zig-zag of weak equivalences
\begin{equation} \label{zz}|\Gr (S. A)| \srm{\heq} |\Omega B \Gr (S. A)| \stackrel{\heq}{\longleftarrow} |\Omega B (S. A)| \srm{\heq} \Omega | B S. A| \srm{\heq} \Omega BA.
\end{equation}
We will explain the intermediate simplicial spaces as we go.

For each $k$, $\Gr (S_k A)$ is a discrete abelian group, and hence 
the natural maps 
$$\gamma\co \Gr (S_k A) \maps \Omega B \Gr (S_k A)$$
are weak equivalences. Letting $k$ vary, we obtain the first weak equivalence of simplicial spaces in (\ref{zz}).

Next, since  $S_k A$ is good, the natural map $S_k A \to \Gr (S_k A)$ induces a weak equivalence
$$\Omega BS_k A \srm{\heq} \Omega B \Gr (S_k A)$$
by Quillen~\cite[Appendix Q, Propositions Q.1, Q.2]{Friedlander-Mazur}. 
These maps assemble to give the second weak equivalence in (\ref{zz}).

The third weak equivalence in (\ref{zz}) is the natural map from the level-wise loop space of a simplicial space to the loop space of its realization; since $B S_k A$ is connected for each $k$, this map 
is a weak equivalence by May~\cite[Theorem 12.3]{May-GOILS}.\footnote{We are using the thin geometric realization in this section so that we can apply May's result.}

Finally, the last map in (\ref{zz}) is induced by the weak equivalence 
$$B S. A \srm{\heq} BA$$
obtained by viewing $B S. A$ as the geometric realization of the simplicial space $k \mapsto |(S. A)^k| \isom |S. (A^k)|$, which admits a level-wise weak equivalence to the (proper) simplicial space $k\goesto A^k$ underlying $BA$.

To complete the proof, we will show that 
there is natural weak equivalence of simplicial sets
\begin{equation}\label{a1}\colim_\mathcal{B} S. A \srm{\phi} \Gr (S. A).\end{equation}
 Let $t\co F(\mathcal{B}) \to A$ be the monoid homomorphism induced by the inclusion $\mathcal{B}\injects A$.
Elements in $\colim_\mathcal{B} S_k A$ are represented by pairs $(b, \sigma)$ with $b \in F(\mathcal{B})$, $\sigma\co \Delta^k\to A$, and we set
$$\phi([(b, \sigma)]) = \sigma - c_{t(b)} \in \Gr (S_k A),$$
where $c_{t(b)}$ is the degenerate $k$--simplex with value $t(b) \in A$. 
 Since  $S. A$ is good,  the main result of Quillen~\cite[Appendix Q]{Friedlander-Mazur}  shows that $\phi$ restricts to a homology equivalence between the connected components of the identity elements
(namely the element in  $\colim_\mathcal{B} S_k A$ represented by $(0, c_0)$ and the element in  $c_0 \in \Gr (S_0 A)$).
We will show that the domain and range of $\phi$ are group-like topological abelian monoids, with $\phi$ inducing an isomorphism between their groups of connected components.  Then $\phi$ is a homology isomorphism between simple spaces, hence a weak equivalence~\cite[Proposition 4.74, Example 4A.3]{Hatcher}.

Since  $|\Gr (S. A)|$ is the realization of a simplicial abelian group, it is a topological abelian group. 
Next, we claim that $\colim_\mathcal{B} S. A$ has the structure of a simplicial abelian monoid, so that 
its geometric realization is a topological abelian monoid. For each $k$, the  addition map $A\cross A \to A$ induces a map 
$$ (S_k A) \cross (S_k A) \isom S_k (A\cross A)\maps S_k A$$
 and since $A$ is (strictly) abelian, there is a well-defined map
$$\colim_\mathcal{B} S_k A \cross \colim_\mathcal{B} S_k A \maps \colim_A S_k A,$$
induced by 
$$(b, \sigma)+(b', \sigma') = (b+b', \sigma+\sigma').$$ 
 This map makes $\colim_\mathcal{B} S_k A$ a monoid (with identity element represented by $(0, c_0)$). Moreover, the simplicial  structure maps in $\colim_\mathcal{B} S. A$ are monoid homomorphisms.  

The map 
\begin{equation}\label{pi0-ab}\pi_0 |\colim_\mathcal{B} S. A| \to \pi_0 |\Gr (S. A)|\end{equation}
is induced by a map of simplicial monoids, so it is a monoid homomorphism. 
We have natural isomorphisms of monoids
$$\pi_0 |\colim_\mathcal{B} S. A|    \isom \colim_\mathcal{B} \pi_0 |S. A| \isom  \colim_\mathcal{B} \pi_0 A
\underset{\isom}{\overset{f}{\maps}}
\Gr (\pi_0 A),$$
where $f$ sends the point in $ \colim_\mathcal{B} \pi_0 A$ represented by $(b, [a])$ to $[a]-[t(b)]\in \Gr (\pi_0 A)$.
Hence the domain of (\ref{pi0-ab}) is a group.
Let
$$g\co \Gr (\pi_0 A) \srm{\isom} \pi_0 \Omega BA$$
be the natural isomorphism sending $[a'] - [a]$ to the path component of the loop $\gamma(a') \cc \ol{\gamma(a)}$. 
Letting $h$ denote the bijection $\pi_0 (\Omega BA)\srt{\isom} \Gr (S. A)$ induced by the zig-zag of weak equivalences (\ref{zz}) constructed above
tracing through the  definitions (and applying Lemma~\ref{loops}) shows that (\ref{pi0-ab}) agrees with the bijection $h\circ g\circ f$. 
This shows that (\ref{pi0-ab}) is an isomorphism of groups, completing the proof.
\end{proof}

\begin{corollary} Let $A$ be a proper,   topological  abelian monoid with base system $\mathcal{B}$. Then for each $k\geqs 1$, there is a natural isomorphism
$$\pi_k (\Omega BA) \isom \colim_{b\in F(\mathcal{B})} \pi_k (A, t(b)) \isom [S^k, A]/\pi_0 A.$$
\end{corollary}

The colimit above refers to the functor from $F(\mathcal{B})$ to groups sending $b\in F(\mathcal{B})$ to $\pi_k (A, t(b))$, and sending each morphism labeled by $c$ to the map induced by multiplication by $t(c)$.

\begin{proof} The first isomorphism follows from Theorem~\ref{Ab-thm}. Indeed,
letting $|c_a|$ denote the point in $|S.A|$ corresponding to the $0$--simplex at $a\in A$,
 we have
$$\pi_k (\colim_{b\in F(\mathcal{B})} |S. A|, [(0, |c_0|)]) \isom \colim_{b\in F(\mathcal{B})} \pi_k (|S. A|, |c_{t(b)}|) \isom \colim_{b\in F(\mathcal{B})} \pi_k (A, t(b)),$$
where the second isomorphism follows from the fact that the weak equivalence $|S. A|\to A$ sends $|c_a|$ to $a$ and induces a map between colimits.
(Note that here we are taking a colimit in the category of groups.)

Consider the map
$$\psi\co  \colim_{b\in F(\mathcal{B})} \pi_k (A, t(b)) \maps  [S^k, A]/\pi_0 A$$
defined by  $\psi(b, \la \alpha \ra) = [\alpha]$. This map is well-defined, and Lemma~\ref{EH} shows that it is a monoid homomorphism. Surjectivity of $\psi$ follows from the fact that  every $\alpha\co S^k \to A$ is homotopic to a map $\alpha'$ with $\alpha'(1)\in \mathcal{B}$. To check injectivity of $\psi$, note that if $[\alpha]$ is zero in $[S^k, A]/\pi_0 A$, then there exists a constant map $c$ such that $\alpha+c$ is nullhomotopic. Now  $[c] = [c']$ for some $c'\in \mathcal{B}$, and 
$$(\alpha(1), \la \alpha \ra) \equiv (\alpha(1)+c', \la \alpha + c'\ra)$$
in $\colim_{b\in F(\mathcal{B})} \pi_k (A, t(b))$.
Since $\alpha+c' \heq \alpha+c$ is nullhomotopic, we have $\la \alpha + c'\ra = 0\in \pi_k (A, \alpha(1)+c')$.
\end{proof}

 \section{Connected components of the homotopy group completion}\label{Appendix}
 
We give a short proof of the following  well-known result.
    
 \begin{proposition} Let $M$ be a topological monoid. 
Then
the natural map $\gamma\co M\to \Omega BM$ induces an isomorphism 
$$\Gr(\pi_0 M)  \srm{\isom} \pi_0 (\Omega BM)$$
\end{proposition}

As in the body of the paper, we work in the category of compactly generated spaces, and we form $BM$ using the thick geometric realization.
However, if $M$ is not compactly generated, the identity map induces an isomorphism
 $\Gr(\pi_0 M)\isom \Gr(\pi_0 (k(M)))$, where $k(M)$ denotes $M$ with the compactly generated topology, so the isomorphism
 $$\Gr (\pi_0 M) \isom  \pi_0 (\Omega B k(M))$$
 holds even when $M$ is not compactly generated.

 \begin{proof}
Following the method in Section~\ref{sec:proof}, we consider the commutative diagram
$$
\xymatrix{
M\ar[r]^-\gamma & \Omega BM \\
||S. M|| \ar[r]^-\gamma \ar[u]^\heq & \Omega B ||S.M||. \ar[u]^\heq
}
$$
The vertical map on the left is a weak equivalence by~\cite[Lemma 1.11]{ERW}.
The vertical map on the right is a weak equivalence by~\cite[Theorem 2.2]{ERW},  because it is induced by a level-wise weak equivalence (here it is important that we use the thick realization if $M$ is not proper).
After applying the functor $\Gr (\pi_0 (-))$ to the above diagram, we see that it suffices to show that the map
$$\gamma_* \co \Gr ( \pi_0 ||S. M|| ) \maps \pi_0 ( \Omega B S.M)\isom \pi_1 B S.M$$
induced by $\gamma$ is an isomorphism. The inclusion of the constant simplices into $S.M$ induces an isomorphism $\pi_0 M\xmaps{\isom} \pi_0  ||S.M||$, so we just need to show that the induced map
\begin{equation}\label{eq:app}g\co \Gr (\pi_0 M)\maps \pi_1  B S.M\end{equation}
is an isomorphism. 

As in Section~\ref{sec:proof}, we view $B S.M$ as the diagonal of a bisimplicial set. As such,  $B S.M$ has a single 0--simplex; its set of  1--simplices is $S_1 M$; and its set of two--simplices is $S_2 (M\cross M) \isom S_2  M \cross S_2 M$. 
The associated CW structure on $B S.M$ yields a presentation  for $\pi_1  B S.M$ with the 1--simplices as generators 
and the 2--simplices as relations. We will simply write $\lra{\sigma}$ for the element in $\pi_1 B S.M$ associated to the singular 1--simplex 
$$\sigma \co [0,1]\to M$$
 in $S_1 M$. Each 2--simplex is represented a pair of singular 2--simplices 
 $$\tau, \nu \co \Delta^2\to M,$$ 
 and  the 2--simplex
$(\tau, \nu)$  
yields the relation 
$$\lra{d_2 \tau}\lra{d_0 \nu} = \lra{(d_1 \tau)\bullet (d_1 \nu)},$$
where the $d_i$ are the face operators in the simplicial set $S. M$ and
$\bullet$ denotes the point-wise product of maps into $M$.

Let $F(S_1 M)$ denote the free group on the set $S_1 M$.
Consider the homomorphism
$$f\co F(S_1 M) \maps \Gr (\pi_0 M)$$
defined by $f(\sigma) = [\sigma(0)]$. 
Since the 2--simplex is path connected, we see that $f$ sends all relators
$(d_2 \tau)\cdot (d_0 \nu)\cdot \left((d_1 \tau)\bullet (d_1 \nu)\right)^{-1}$ to the identity. Hence
 $f$ descends to a homomorphism out of $\pi_1 B S.M$ (which we continue to denote by $f$). Moreover, letting $u\co \pi_0 M \to \Gr (\pi_0 M)$ denote the the universal map from $\pi_0 M$ to its group completion, we see that
 $$\pi_0 M \srm{u} \Gr (\pi_0 M) \srm{g} \pi_1 B S.M \srm{f} \Gr (\pi_0 M)$$
is simply $u$, and the universal property of $u$ implies that $f\circ g$ must be the identity map on $\Gr (\pi_0 M)$.
 
 To complete the proof, we will show that $g\circ f$ is the identity on $\pi_1 B S.M$. For each generator $\sigma\in S_1 M$ we have 
 $g\circ f ([\sigma]) = \lra{\sigma(0)}$, where $\sigma (0)$ is the constant map $[0,1]\to M$ with value $\sigma (0)$. 
 We must show that $\lra{\sigma}=\lra{\sigma(0)}$ in $\pi_1 B S.M$.
Let $\ol{\sigma}$ denote the reverse of $\sigma$ (so $\ol{\sigma}(t) = \sigma(1-t)$).
 A based nullhomotopy of the loop $\sigma \cc \ol{\sigma}$ yields a singular 2--simplex $\tau\co \Delta \to M$ with $d_0 \tau = \ol{\sigma}$, $d_2 (\tau) = \sigma$, and $d_1 (\tau) = \sigma(0)$. Now
 consider the
 2--simplex $(\tau, s_0 \ol{\sigma})\in S_2 (M\cross M)$, where $s_0$ denotes the degeneracy operator in $S. M$.
This yields the relation $\lra{\sigma}\lra{\ol{\sigma}} = \lra{\sigma(0) \bullet \ol{\sigma}}$ in $\pi_1 B S.M$. Since  $\pi_1  B S.M$ is a group, to complete the proof it  suffices to show that $\lra{\sigma(0) \bullet \ol{\sigma}} = \lra{ \sigma(0)}\lra{\ol{\sigma}}$. 
More generally, for any $\alpha, \beta\in S_1 M$, the relation $\lra{\alpha\bullet \beta}= \lra{\alpha}\lra{\beta} $
is witnessed by the 2--simplex $(s_1 \alpha, s_0 \beta) \in S_2 (M\cross M)$.
\end{proof}
 
\def\cprime{$'$}


\begin{thebibliography}{10}

\bibitem{Baird-Ramras-arxiv-v3}
Thomas Baird and Daniel~A. Ramras.
\newblock Smooth approximation in algebraic sets and the topological
  {A}tiyah--{S}egal map.
\newblock arXiv:1206.3341v3, 2013.

\bibitem{Bergeron}
Maxime Bergeron.
\newblock The topology of nilpotent representations in reductive groups and
  their maximal compact subgroups.
\newblock {\em Geom. Topol.}, 19(3):1383--1407, 2015.

\bibitem{Dolgachev}
Igor Dolgachev.
\newblock {\em Lectures on invariant theory}, volume 296 of {\em London
  Mathematical Society Lecture Note Series}.
\newblock Cambridge University Press, Cambridge, 2003.

\bibitem{DGM}
Bj{\o}rn~Ian Dundas, Thomas~G. Goodwillie, and Randy McCarthy.
\newblock {\em The local structure of algebraic {K}-theory}, volume~18 of {\em
  Algebra and Applications}.
\newblock Springer-Verlag London, Ltd., London, 2013.

\bibitem{ERW}
Johannes Ebert and Oscar Randal-Williams.
\newblock Semi-simplicial spaces.
\newblock \href{https://arxiv.org/abs/1705.03774}{arXiv:1705.03774}, 2017.

\bibitem{Eckmann-Hilton}
B.~Eckmann and P.~J. Hilton.
\newblock Group-like structures in general categories. {I}. {M}ultiplications
  and comultiplications.
\newblock {\em Math. Ann.}, 145:227--255, 1961/1962.

\bibitem{EKMM}
A.~D. Elmendorf, I.~Kriz, M.~A. Mandell, and J.~P. May.
\newblock {\em Rings, modules, and algebras in stable homotopy theory},
  volume~47 of {\em Mathematical Surveys and Monographs}.
\newblock American Mathematical Society, Providence, RI, 1997.
\newblock With an appendix by M. Cole.

\bibitem{FL-free}
Carlos Florentino and Sean Lawton.
\newblock The topology of moduli spaces of free group representations.
\newblock {\em Math. Ann.}, 345(2):453--489, 2009.

\bibitem{FL-abelian}
Carlos Florentino and Sean Lawton.
\newblock Topology of character varieties of {A}belian groups.
\newblock {\em Topology Appl.}, 173:32--58, 2014.

\bibitem{FLR}
Carlos Florentino, Sean Lawton, and Daniel Ramras.
\newblock Homotopy groups of free group character varieties.
\newblock {\em Ann. Sc. Norm. Super. Pisa Cl. Sci. (5)}, 17(1):143--185, 2017.

\bibitem{Friedlander-91}
Eric~M. Friedlander.
\newblock Algebraic cycles, {C}how varieties, and {L}awson homology.
\newblock {\em Compositio Math.}, 77(1):55--93, 1991.

\bibitem{FL}
Eric~M. Friedlander and H.~Blaine Lawson, Jr.
\newblock A theory of algebraic cocycles.
\newblock {\em Ann. of Math. (2)}, 136(2):361--428, 1992.

\bibitem{Friedlander-Mazur}
Eric~M. Friedlander and Barry Mazur.
\newblock Filtrations on the homology of algebraic varieties.
\newblock {\em Mem. Amer. Math. Soc.}, 110(529):x+110, 1994.
\newblock With an appendix by Daniel Quillen.

\bibitem{Hatcher}
Allen Hatcher.
\newblock {\em Algebraic topology}.
\newblock Cambridge University Press, Cambridge, 2002.

\bibitem{Hironaka}
Heisuke Hironaka.
\newblock Triangulations of algebraic sets.
\newblock In {\em Algebraic geometry (Proc. Sympos. Pure Math., Vol. 29,
  Humboldt State Univ., Arcata, Calif., 1974)}, pages 165--185. Amer. Math.
  Soc., Providence, R.I., 1975.

\bibitem{Lawson-homology}
H.~Blaine Lawson, Jr.
\newblock The topological structure of the space of algebraic varieties.
\newblock {\em Bull. Amer. Math. Soc. (N.S.)}, 17(2):326--330, 1987.

\bibitem{Lawson-thesis}
Tyler Lawson.
\newblock {\em Derived representation theory of nilpotent groups}.
\newblock ProQuest LLC, Ann Arbor, MI, 2004.
\newblock Thesis (Ph.D.)--Stanford University.

\bibitem{Lawson-prod}
Tyler Lawson.
\newblock The product formula in unitary deformation {$K$}-theory.
\newblock {\em $K$-Theory}, 37(4):395--422, 2006.

\bibitem{Lawson-simul}
Tyler Lawson.
\newblock The {B}ott cofiber sequence in deformation {$K$}-theory and
  simultaneous similarity in {${\rm U}(n)$}.
\newblock {\em Math. Proc. Cambridge Philos. Soc.}, 146(2):379--393, 2009.

\bibitem{Lewis-thesis}
L.~Gaunce Lewis, Jr.
\newblock {\em The Stable Category and Generalized Thom Spectra}.
\newblock 1978.
\newblock Thesis (Ph.D.)--The University of Chicago.

\bibitem{LF-survey}
Paulo Lima-Filho.
\newblock Topological properties of the algebraic cycles functor.
\newblock In {\em Transcendental aspects of algebraic cycles}, volume 313 of
  {\em London Math. Soc. Lecture Note Ser.}, pages 75--119. Cambridge Univ.
  Press, Cambridge, 2004.

\bibitem{LM}
Alexander Lubotzky and Andy~R. Magid.
\newblock Varieties of representations of finitely generated groups.
\newblock {\em Mem. Amer. Math. Soc.}, 58(336):xi+117, 1985.

\bibitem{May-GOILS}
J.~P. May.
\newblock {\em The geometry of iterated loop spaces}.
\newblock Lecture Notes in Mathematics, Vol. 271. Springer-Verlag, Berlin-New
  York, 1972.

\bibitem{McDuff-Segal}
D.~McDuff and G.~Segal.
\newblock Homology fibrations and the ``group-completion'' theorem.
\newblock {\em Invent. Math.}, 31(3):279--284, 1975/76.

\bibitem{Miller-Palmer}
Jeremy Miller and Martin Palmer.
\newblock A twisted homology fibration criterion and the twisted
  group-completion theorem.
\newblock {\em Q.~J.~Math.}, 66(1):265--284, 2015.

\bibitem{Milnor-realization}
John Milnor.
\newblock The geometric realization of a semi-simplicial complex.
\newblock {\em Ann. of Math. (2)}, 65:357--362, 1957.

\bibitem{Peters-Kosarew}
Chris Peters and Siegmund Kosarew.
\newblock Introduction to {L}awson homology.
\newblock In {\em Transcendental aspects of algebraic cycles}, volume 313 of
  {\em London Math. Soc. Lecture Note Ser.}, pages 44--71. Cambridge Univ.
  Press, Cambridge, 2004.

\bibitem{RWY}
Daniel Ramras, Rufus Willett, and Guoliang Yu.
\newblock A finite-dimensional approach to the strong {N}ovikov conjecture.
\newblock {\em Algebr. Geom. Topol.}, 13(4):2283--2316, 2013.

\bibitem{Ramras-excision}
Daniel~A. Ramras.
\newblock Excision for deformation {$K$}-theory of free products.
\newblock {\em Algebr. Geom. Topol.}, 7:2239--2270, 2007.

\bibitem{Ramras-stable-moduli}
Daniel~A. Ramras.
\newblock The stable moduli space of flat connections over a surface.
\newblock {\em Trans. Amer. Math. Soc.}, 363(2):1061--1100, 2011.

\bibitem{Ramras-crystal}
Daniel~A. Ramras.
\newblock Periodicity in the stable representation theory of crystallographic
  groups.
\newblock {\em Forum Math.}, 26(1):177--219, 2014.

\bibitem{Ramras-TAS}
Daniel~A. Ramras.
\newblock The topological {A}tiyah--{S}egal map.
\newblock \href{https://arxiv.org/abs/1607.06430}{arXiv:1607.06430}, 2016.

\bibitem{Randal-Williams-scholium}
Oscar Randal-Williams.
\newblock `{G}roup-completion', local coefficient systems and perfection.
\newblock {\em Q. J. Math.}, 64(3):795--803, 2013.

\bibitem{Segal-cat-coh}
Graeme Segal.
\newblock Categories and cohomology theories.
\newblock {\em Topology}, 13:293--312, 1974.

\bibitem{Wang-realization}
Yi-Sheng Wang.
\newblock Geometric realization and its variants.
\newblock \href{https://arxiv.org/abs/1804.00345}{arXiv:1804.00345}, 2018.

\end{thebibliography}
\end{document}